\documentclass[11pt]{amsart}
\usepackage{amsmath,amssymb,latexsym,esint,cite,mathrsfs}
\usepackage{verbatim,wasysym,cite}
\usepackage[left=2.7cm,right=2.7cm,top=2.7cm,bottom=2.7cm]{geometry}
\usepackage[latin1]{inputenc}
\usepackage{microtype}
\usepackage{color,enumitem,graphicx}
\usepackage[colorlinks=true,urlcolor=blue, citecolor=red,linkcolor=blue,
linktocpage,pdfpagelabels, bookmarksnumbered,bookmarksopen]{hyperref}
\usepackage[hyperpageref]{backref}
\usepackage[english]{babel}
\renewcommand{\epsilon}{{\varepsilon}}

\numberwithin{equation}{section}
\newtheorem{theorem}{Theorem}[section]
\newtheorem{lemma}[theorem]{Lemma}
\newtheorem{remark}[theorem]{Remark}

\newtheorem{proposition}[theorem]{Proposition}

\title[Inhomogeneous Gross-Pitaevskii  equation]
{Some qualitative studies of the focusing inhomogeneous Gross-Pitaevskii  equation}

\author[A. H. Ardila]{Alex H. Ardila}
\address[A. H. Ardila]{ICEx, Universidade Federal de Minas Gerais, Av. Antonio Carlos, 6627, 
Caixa Postal 702, 30123-970, Belo Horizonte-MG, Brazil}
\email{ardila@impa.br}

\author[V. D. Dinh]{Van Duong Dinh}
\address[V. D. Dinh]{Institut de Math\'ematiques de Toulouse UMR5219, Universit\'e Toulouse CNRS, 31062 Toulouse Cedex 9, France and Department of Mathematics, HCMC University of Pedagogy, 280 An Duong Vuong, Ho Chi Minh, Vietnam}
\email{dinhvan.duong@math.univ-toulouse.fr}


\begin{document}

\begin{abstract}
We study the Cauchy problem for an inhomogeneous  Gross-Pitaevskii  equation. We first derive a sharp threshold for global existence and blow up of the solution. Then we construct and classify finite time blow up solutions at the minimal mass threshold.  Additionally, using variational techniques, we study the existence, the orbital stability and instability of standing waves.
\end{abstract}

\maketitle

\section{Introduction}
\label{S:0}
 In this paper, we give some results concerning the Cauchy problem and the dynamics for an nonlinear inhomogeneous  Gross-Pitaevskii  equation in the following form:
\begin{equation}\label{GP}
 \begin{cases} 
 i\partial_{t}u+\Delta u-\gamma^{2}|x|^{2}u+ |x|^{-b}|u|^{p-1}u=0,\\
u(x,0)=u_{0},
\end{cases} 
\end{equation}
where $\gamma>0$,  $u=u(x,t)$ is a complex-valued function of $(x,t)\in \mathbb{R}^{N}\times\mathbb{R}$, $N\geq 1$, $0<b<\min\left\{2, N\right\}$ and $1<p<2^{\circ}$. Here, $2^{\circ}$ is defined by ${2}^{\circ}=1+\frac{4-2b}{N-2}$ if $N\geq 3$, and $2^{\circ}=\infty$ if $N=1$, $2$.

The Schr\"odinger equation \eqref{GP} is a model from various physical contexts in the description of nonlinear
waves such as propagation of a laser beam in the optical fiber. In particular, it models the Bose-Einstein condensates with the attractive interparticle interactions under a magnetic trap. The operator $-\gamma^{2}|x|^{2}$ is the isotropic harmonic potential modelling a magnetic field whose role is to confine the movement of particles.  The inhomogeneous nonlinearity $|x|^{-b}|u|^{p-1}u$ describes the attractive interaction between particles. When $b>0$, it can be thought of as modeling inhomogeneities in the medium in which the wave propagates;
we refer the readers to \cite{AGR,BP} for more information on the related physical backgrounds. In recent years, this type
of equations has attracted attention of numerous researchers due to their significance in theory and applications, see \cite{JC,CGO,  FG1, FGCS, SZ, TSaa, ABRF, VCGG, Dinh}.

In the absence of the harmonic potential, i.e., \eqref{GP} with $\gamma=0$, we refer the reader to \cite{LGF, FG1, FGCS, MGM, SZ, ABRF, VCGG, Dinh} for more information.  In the classical case $b=0$, many authors have been studying the problem of stability of standing waves, see \cite{Carles2002, FOG, STW, Fukui2000, JZ2000, JZ2SS005}. On the other hand, if $\gamma>0$ and $b<0$  the problem \eqref{GP} was treated in \cite{JC, CGO, Luo2018, Chen2010, Saanouni2015}. If $\gamma>0$ and $b>0$, to the best of our knowledge, there are no results concerning the Cauchy problem and the dynamics for \eqref{GP}.

By \cite[Appendix K]{FGCS} and \cite[Theorem 9.2.6]{CB} we can get the time local well-posedness for the Cauchy problem to \eqref{GP} in the space 
\begin{equation*} 
\Sigma(\mathbb{R}^{N}):=\left\{u\in H^{1}(\mathbb{R}^{N}):|x|^{}u\in L^{2}(\mathbb{R}^{N}) \right\},
\end{equation*}
equipped with the norm 
\begin{equation*}
\|u\|^{2}_{\Sigma}=\int_{\mathbb{R}^{N}}\left(\left|\nabla u\right|^{2}+|x|^{2}|u|^{2}\right)dx.
\end{equation*}
More precisely, we have the following proposition.
\begin{proposition}\label{WP} 
For every $u_{0}\in \Sigma(\mathbb{R}^{N})$ there exists a unique maximal solution of Cauchy problem \eqref{GP}, $T\in(0, \infty]$, such that $u(0)=u_{0}$ and $u\in C([0,T), \Sigma(\mathbb{R}^{N}))$.  If $T=\infty$, $u$ is called a global solution. If $T<\infty$,  $u$ is called blow-up in finite time and $\lim_{t\rightarrow T} \|\nabla u(t)\|^{2}_{L^{2}}=\infty$. Moreover, we have the conservation of energy and charge: for every $t\in [0,T)$,
\begin{equation*}
E(u(t))=E(u_{0}) \quad and \quad \|u(t)\|^{2}_{L^{2}}=\|u_{0}\|^{2}_{L^{2}},
\end{equation*}
where
\begin{equation}\label{Es}
E(u)=\frac{1}{2}\int_{\mathbb{R}^{N}}|\nabla u|^{2}dx+\frac{\gamma^{2}}{2}\int_{\mathbb{R}^{N}}|x|^{2}|u|^{2}dx-\frac{1}{p+1}\int_{\mathbb{R}^{N}}|x|^{-b}|u|^{p+1}dx.
\end{equation}
\end{proposition}

We remark that if $1<p<1+\frac{4-2b}{N}$, then we have the global existence of Cauchy problem \eqref{GP} in $\Sigma(\mathbb{R}^{N})$. Indeed, let $u$ be a solution of \eqref{GP} as in Proposition \ref{WP}. From Gagliardo-Nirenberg inequality ( see \cite{FGCS, LGF}) 
\begin{equation}\label{GNI}
\int_{\mathbb{R}^{N}}|x|^{-b}|u|^{p+1}dx\leq C\|\nabla u\|^{\frac{N(p-1)}{2}+b}_{L^{2}}\|u\|^{p+1-\frac{N(p-1)}{2}-b}_{L^{2}}, 
\end{equation} 
we have that 
\begin{equation*}
E(u(t))\geq \|\nabla u(t)\|^{2}_{L^{2}}\left(\frac{1}{2}-C\|\nabla u(t)\|^{\frac{N(p-1)}{2}+b-2}_{L^{2}}\|u(t)\|^{p+1-\frac{N(p-1)}{2}-b}_{L^{2}}\right).
\end{equation*} 
Since $1<p<1+\frac{4-2b}{N}$, in view of the conservation of energy and charge, we see that $\|\nabla u(t)\|^{2}_{L^{2}}$ is bounded; that is,  \eqref{GP} is globally well-posed.

On the other hand, assume that $p\geq 1+\frac{4-2b}{N}$ and let $u_{0}\in \Sigma(\mathbb{R}^{N})$. From Lemma \ref{L1} below we see that if $E(u_{0})<0$, then the solution $u$ of the Cauchy problem \eqref{GP} corresponding to $u_{0}$ blows up in finite time.

In the case $p=1+\frac{4-2b}{N}$, we are motived to investigate a sharp sufficient conditions of global existence to the solutions of the Cauchy problem \eqref{GP}. Let $Q$ be denote the unique (up to symmetries) positive radial solution of the following elliptic equation (see \cite{ABRF, FG1})
\begin{equation}\label{Izx}
-\Delta Q+Q- |x|^{-b}|Q|^{\frac{4-2b}{N}}Q=0.
\end{equation} 
From \cite{FG1}, we have that $\frac{N}{2+N-b}\|Q\|^{\frac{4-2b}{N}}_{L^{2}}$ is the minimum of Weinstein functional
\begin{equation*}
J(u)=\left[\|\nabla u\|^{2}_{L^{2}}\| u\|^{\frac{4-2b}{N}}_{L^{2}}\right]\div \int_{\mathbb{R}^{N}}|x|^{-b}|u|^{\frac{4-2b}{N}+2}dx.
\end{equation*} 
Following the argument of Zhang \cite{ZEc}, we have

\begin{theorem}\label{SC}
Let $p=1+\frac{4-2b}{N}$. Assume that $u_{0}\in \Sigma(\mathbb{R}^{N})$.\\
(i) If $u_{0}$ satisfies  $\|u_{0}\|_{L^{2}}< \|Q\|_{L^{2}}$, then the corresponding solution $u(x,t)$ of the Cauchy problem \eqref{GP} given in Proposition \ref{WP} exists globally in the time.\\
(ii) For arbitrary positive $\lambda$ and complex number $c$  satisfying $|c|\geq1$,  if we take initial data $u_{0}=c\lambda^{\frac{N}{2}} Q(\lambda x)$, then $\|u_{0}\|_{L^{2}}\geq\|Q\|_{L^{2}}$ and the corresponding solution $u(x,t)$ of the Cauchy problem \eqref{GP} blows up in finite time.
\end{theorem}

Notice that if $1<p<1+\frac{4-2b}{N}$, then we have global well-posedness of the Cauchy problem \eqref{GP}. On the other hand, from Theorem \ref{SC}, when $p= 1+\frac{4-2b}{N}$,  all solutions with a mass strictly below that $\|Q\|_{L^{2}}$ are global. If the mass is greater than or equal to $\|Q\|_{L^{2}}$ there are collapse solutions to exists for Eq. \eqref{GP}. So, in the case $p= 1+\frac{4-2b}{N}$, we call $\|Q\|_{L^{2}}$ the critical mass for \eqref{GP}. 

Let us consider the function
\begin{equation}\label{AX}
S_{\beta, \theta_{0}}(x, t)=e^{i\theta_{0}}e^{i\frac{\beta^{2}}{t}}e^{-i\frac{|x|^{2}}{4t}}\left(\frac{\beta}{t}\right)^{\frac{N}{2}}Q \left(\frac{\beta x}{t}\right),
\end{equation} 
where $t>0$, $\beta$, $\theta_{0}\in \mathbb{R}$  and $Q$ is defined by \eqref{Izx}.

In the next result,  inspired by the work of R. Carles \cite{RCIH},  we classify finite time blow-up solutions at the minimal mass threshold.

\begin{theorem}\label{CMm23}
Let $p=1+\frac{4-2b}{N}$ and $\gamma>0$. Assume that $u$ is  a critical mass solution of \eqref{GP} which blows up in finite time $0<T< \frac{\pi}{4\gamma}$, that is, $\|u_{0}\|_{L^{2}}=\|Q\|_{L^{2}}$ and $\lim_{t\rightarrow T}\|\nabla u(t)\|_{L^{2}}=+\infty$. \\
Then there exist $\theta_{0}\in \mathbb{R}$ and $\lambda_{0}>0$ such that 
\begin{equation*}
u(t)=\left(\frac{1}{\mathrm{cos}\,2\gamma t}\right)^{\frac{N}{2}}e^{-i\frac{\gamma}{2}|x|^{2}\mathrm{tan}\,2\gamma t}S_{{\lambda_{0}}, \theta_{0}}\left(\frac{x}{\mathrm{cos}\,2\gamma t}, \frac{\mathrm{sin}\,2\gamma(T-t)}{2\gamma\mathrm{cos}\,2\gamma T\,\, \mathrm{cos}\,2\gamma t}\right)
\end{equation*} 
for every $t\in [0,T)$, where $S_{\lambda_{0}, \theta_{0}}$ is defined in \eqref{AX}. In particular, with the change of variable $\beta_{0}=\lambda_{0}\,\mathrm{cos}\,2\gamma T$, we see that the initial data is of the form
\begin{equation*}
u_{0}(x)=e^{ i\,\theta_{0}}e^{4\gamma\frac{\beta_{0}}{\mathrm{sin}\,4\gamma T}}e^{-i\frac{\gamma}{2}|x|^{2}\mathrm{cot}\,2\gamma T}\left(\frac{2\gamma \,\beta_{0}}{\mathrm{sin}\,2\gamma T}\right)^{\frac{N}{2}}Q\left(\frac{2\gamma \,\beta_{0}}{\mathrm{sin}\,2\gamma T}\,x\right).
\end{equation*} 
\end{theorem}

\begin{remark} 
(i) Obviously we can prove similar result as Theorem \ref{CMm23} also in the case where $u$ is  a critical mass solution of \eqref{GP} which blows up in the past, i.e., for  $-{\pi}/{4\gamma}<T<0$.\\
(ii)Let $u$ satisfy the hypotheses of the Theorem \ref{CMm23}. Since $Q$ is spherically symmetric, it is not difficult to show that the function $u$ satisfies the relation $u(x, t+\frac{n\pi}{2\gamma})=u(x, t)$ for every $n\in \mathbb{N}$ and $0\leq t<T$. This implies,  by using a time-translation and (i),  that if $u(t)$ does not collapse in finite time $0<T<{\pi}/{2\gamma}$, then it will never collapse in the future. 
\end{remark}




By a standing wave, we mean a solution of \eqref{GP} with the form $u(x,t)=e^{i\omega t}\varphi(x)$ with $\omega\in \mathbb{R}$ and $\varphi$ satisfying the following nonlinear elliptic problem
\begin{equation}\label{Sp}
 \begin{cases} 
 -\Delta \varphi+\omega\varphi+\gamma^{2}|x|^{2}\varphi-|x|^{-b}|\varphi|^{p-1}\varphi=0,\\
\varphi\in \Sigma(\mathbb{R}^{N})\setminus \left\{0\right\}.
\end{cases} 
\end{equation}
We  remember that $\lambda_{1}=\gamma\, N$ is the simple first eigenvalue of the many-dimensional harmonic oscillator $-\Delta +\gamma^{2}|x|^{2}$. More precisely,
\begin{equation}\label{INF1}
\gamma N=\inf\left\{\|\nabla u\|_{L^{2}}^{2}+\gamma^{2}\|xu\|_{L^{2}}^{2}:u\in \Sigma(\mathbb{R}^{N}), \| u\|_{L^{2}}^{2}=1 \right\}.
\end{equation}
The corresponding eigenfunction to $\lambda_{1}$ is 
\begin{equation}\label{Efa}
\Phi(x):=\pi^{-\frac{N}{2}}e^{-\gamma\frac{|x|^{2}}{2}}
\end{equation}
and we have the inequality 
\begin{equation}\label{Igst}
\gamma\, N \|u\|^{2}_{L^{2}}\leq \|\nabla u\|_{L^{2}}^{2}+\gamma^{2}\|xu\|_{L^{2}}^{2}.
\end{equation}
Notice that if $\omega\leq-\gamma\, N$, then the problem \eqref{Sp} does not admit positive solutions. Indeed, suppose that $\varphi$ is a positive solution of \eqref{Sp}. After multiplication of \eqref{Sp} by the function $\Phi$ defined above, and integrating, we infer
\begin{equation*}
(\omega+\gamma N)\int_{\mathbb{R}^{N}}\varphi(x)\Phi(x)\, dx=\int_{\mathbb{R}^{N}}|x|^{-b}\varphi^{p}(x)\Phi(x)>0.
\end{equation*}
Thus $\omega>-\gamma\, N$. On the other hand, since $\Sigma (\mathbb{R}^{N})\hookrightarrow L^{r+1}(\mathbb{R}^{N})$ is compact, where $1\leq r<1+4/(N-2)$($N\geq 3$), $1\leq r<\infty$ ($N=1$, $2$),  we have that there is at least one solution $\varphi\in C(\mathbb{R}^{N})\cap C^{2}(\mathbb{R}^{N}\setminus\left\{0\right\})$ of \eqref{Sp} that is spherically symmetric and positive. Indeed,  let $\omega>-\gamma N$. We denote
\begin{align*}
\|u\|^2_{H_\omega}&:= \|\nabla u\|^2_{L^2} + \gamma^2\|xu\|^2_{L^2} + \omega \|u\|^2_{L^2},\\
P(u) &:=\int_{\mathbb{R}^N} |x|^{-b}|u|^{p+1} dx.
\end{align*}
By \eqref{Igst}, we have for every $\omega>-\gamma N$, $\|u\|^2_{H_\omega} \sim \|u\|^2_{\Sigma}$.

We define the following functionals
\begin{align*}
S_\omega(u) &:= E(u) + \frac{\omega}{2}M(u) = \frac{1}{2} \|u\|^2_{H_\omega} - \frac{1}{p+1} P(u), \\
K_\omega(u) &:= \|u\|^2_{H_\omega} - P(u), \\
I(u) &:= \|\nabla u\|^2_{L^2} -\gamma^2 \|xu\|^2_{L^2} -\frac{N(p-1)+2b}{2(p+1)}P(u).
\end{align*}
Note that the elliptic equation \eqref{Sp} can be written as $S'_\omega(\varphi) =0$. We now consider the minimizing problem
\begin{align}
d_\omega := \inf \{S_\omega(u) \ : \ u \in \Sigma \backslash \{0\}, K_\omega(u) =0\} \label{d_ome}
\end{align}
and define the set of minimizers of \eqref{d_ome} by
\begin{align}
\mathcal{M}_{\omega}= \left\{u\in\Sigma \backslash \{0\}: S_\omega(u)=d_\omega, K_\omega(u) =0 \right\}.
\end{align}
We have the following result.
\begin{theorem} \label{prop GS}
Let $\gamma>0$, $N\geq 1$, $0<b<\min \{2,N\}$, $1<p<2^{\circ}$ and $\omega>-\gamma N$. Then $d_\omega>0$ and $d_\omega$ is attained by a function which is a solution to the elliptic equation \eqref{Sp}. Moreover, every minimizer is the form $e^{i\theta}\varphi (x)$, where $\varphi$ a real-valued, positive and spherically symmetric function.
\end{theorem}
Notice that $\varphi$ being radially symmetric, satisfies the ordinary differential equation
	\begin{equation*}
	\varphi^{\prime\prime}+\frac{N-1}{r}\varphi^{\prime}-(\omega+\gamma^{2}r^{2})\varphi+r^{-b}\varphi^{p}=0 \quad \text{in $(0, +\infty)$.}
	\end{equation*}
Using the general results of Shioji and Watanabe \cite{NSKW}, we have that for any $\omega> -\gamma\, N$, $0 < b < 1$, $N\geq 3$ and $1 < p < {2}^{\circ}$  such a solution $\varphi$ is unique, i.e, $\mathcal{M}_{\omega}=\left\{e^{i\theta_{0}}\varphi; \theta_{0}\in\mathbb{R}\right\}$; see Apendix for more details.

We consider the following cross-constrained minimization problem
\[
d_n:= \inf \{ S_\omega(u) \ : \ u \in \mathcal{N}\},
\]
where the constrain $\mathcal{N}$ is given by
\[
\mathcal{N}:= \{ u \in \Sigma \backslash \{0\}, \ K_\omega(u) <0, I(u)=0 \},
\]
and we define
\[
d:= \min\{ d_\omega, d_n\},
\]
where $d_\omega$ is given by \eqref{d_ome}. From Lemma \ref{Povs} we obtain that $d>0$. Now we define the sets
\begin{align*}
K_- &:= \{u \in \Sigma \backslash \{0\} \ : \ S_\omega(u) <d, K_\omega(u)<0, I (u)<0 \}, \\
K_+ &:= \{u \in \Sigma \backslash \{0\} \ : \ S_\omega(u) <d, K_\omega(u)<0, I (u)>0 \}, \\
R_- &:= \{u \in \Sigma \backslash \{0\} \ : \ S_\omega(u) <d, K_\omega(u)<0 \}, \\
R_+ &:= \{u \in \Sigma \backslash \{0\} \ : \ S_\omega(u) <d, K_\omega(u)>0 \}.
\end{align*}
\begin{remark}\label{RM1}
	By the definition, we see that
	\[
	\{u \in \Sigma \backslash \{0\} \ : \ S_\omega<d \} = R_+ \cup K_+ \cup K_-.
	\]
\end{remark}

We are now able to show the sharp threshold for global existence and blow up of solutions to \eqref{GP}. 
\begin{theorem} \label{prop blowup}
Let $\gamma >0, N\geq 1, 0<b<\min \{2,N\}, 1+\frac{4-2b}{N} \leq p<2^{\circ}$ and $\omega>-\gamma N$.\\
(i) If $u_0 \in K_-$, then the corresponding solution to \eqref{GP} blows up in finite time. \\
(ii) If $u_0 \in R_+ \cup K_+$, then the corresponding solution to \eqref{GP} exists globally in time.
\end{theorem}
From Theorem \ref{prop blowup} and Remark \ref{RM1} we infer that if $S_\omega(u_0)<d$, then the solution of Cauchy problem \eqref{GP} exists globally if and only if $u_0\in R_+ \cup K_+$.

From a physical point of view, the most important solutions of the stationary problem \eqref{Sp} are the so-called ground states solutions; that is, which are the minimizers  of the energy functional $E$ subject to a prescribed mass constraint $q>0$,
\begin{equation}\label{PMp}
I_{q}=\inf\left\{E(u),\quad u\in \Sigma(\mathbb{R}^{N}),\quad \|u\|^{2}_{L^{2}}=q \right\}.
\end{equation}
Eventually, we introduce the set of ground states of \eqref{Sp} by
\begin{equation*}
\mathcal{G}_{q}:=\left\{\varphi\in \Sigma(\mathbb{R}^{N}) \quad \text{such that} \quad  I_{q}=E(\varphi), \quad \|\varphi\|^{2}_{L^{2}}=q\right\}.
\end{equation*}
Notice that if $\varphi\in \mathcal{G}_{q}$, then there exists a Lagrange multiplier $\omega\in \mathbb{R}$ such that \eqref{Sp} is satisfied. Thus, $u(x,t)=e^{i\omega t}\varphi(x)$ is a solution of the Cauchy problem \eqref{GP} with initial condition $u_{0}=\varphi$.

We present a result about the existence of ground state.
\begin{theorem}\label{CMm}
Let $\gamma>0$, $N\geq 1$, $0<b<\min\left\{2, N\right\}$  and $1<p<1+\frac{4-2b}{N}$. \\
(i) Any minimizing sequence of $I_{q}$ is relatively compact in $\Sigma(\mathbb{R}^{N})$. In particular, the set of ground states $\mathcal{G}_{q}$ is not empty.\\
(ii) If $u\in \mathcal{G}_{q}$, then there exists a real-valued, positive and spherically symmetric function $\varphi\in \Sigma(\mathbb{R}^{N})$ such that $u(x)=e^{i\theta_{0}}\varphi(x)$ with $\theta_{0}\in \mathbb{R}$.
\end{theorem}

For the critical case $p=1+\frac{4-2b}{N}$,  under appropriate assumption on $q$, we have similar results.
\begin{theorem}\label{CC}
Let $\gamma>0$, , $N\geq 1$, $0<b<\min\left\{2, N\right\}$ and $p=1+\frac{4-2b}{N}$. \\
Let $q$ satisfy that $q< \|Q\|^{2}_{L^{2}}$. Then the set $\mathcal{G}_{q}$ is not empty. Moreover, every minimizer is of the form $e^{i\theta_{0}}\varphi(x)$, where $\varphi$ is a positive and spherically symmetric function and $\theta_{0}\in \mathbb{R}$.
\end{theorem}

Notice that if $p>1+\frac{4-2b}{N}$, then we have $I_{q}=-\infty$. Indeed, for $v\in \Sigma(\mathbb{R}^{N})$ with $\|v\|^{2}_{L^{2}}=q$ we define  $v_{\mu}(x):=\mu^{\frac{N}{2}}v(\mu x)$. It is clear that  $\|v_{\mu}\|^{2}_{L^{2}}=\|v\|^{2}_{L^{2}}$ and 
\begin{equation*}
E(v_{\mu})=\frac{\mu^{2}}{2}\int_{\mathbb{R}^{N}}|\nabla v|^{2}dx+\mu^{-2}\frac{\gamma^{2}}{2}\int_{\mathbb{R}^{N}}|x|^{2}|v|^{2}dx-\frac{\mu^{\frac{N}{2}(p-1)+b}}{p+1}\int_{\mathbb{R}^{N}}|x|^{-b}|u|^{p+1}dx.
\end{equation*}
Thus, since $p>1+\frac{4-2b}{N}$, it follows that $E(v_{\mu})\rightarrow -\infty$ as $\mu$ goes to $+\infty$. To show the existence of ground states in the supercritical case $1+\frac{4-2b}{N}<p<{2}^{\circ}$, we consider a local minimization problem. Following \cite{BEBOJEVI2017}, we introduce the following sets
\begin{align*}
S_{q}:=&\left\{u\in\Sigma(\mathbb{R}^{N}): \|u\|^{2}_{L^{2}}=q \right\} \\ 
B_{r}:=&\left\{u\in\Sigma(\mathbb{R}^{N}): \|u\|^{2}_{H}\leq r \right\}, 
\end{align*}
where $\|\cdot\|_{H}$ denotes the norm 
\begin{equation}\label{Hn}
\|u\|^{2}_{H}:= \|\nabla u\|^{2}_{L^{2}}+\gamma^{2}\|xu\|^{2}_{L^{2}}.
\end{equation}
For a fixed $q>0$ and $r>0$, we set the following local variational problem
\begin{equation}\label{Lmp1}
I^{r}_{q}=\inf\left\{E(u),\quad u\in S_{q}\cap B_{r}\right\}.
\end{equation}
Notice that if $S_{q}\cap B_{r}\neq \emptyset$, then by \eqref{GNI} we infer that $I^{r}_{q}>-\infty$. We denote the set of nontrivial solutions of \eqref{Lmp1} by
\begin{equation*}
\mathcal{G}^{r}_{q}:=\left\{\varphi\in S_{q}\cap B_{r} \quad \text{such that} \quad  I^{r}_{q}=E(\varphi)\right\}.
\end{equation*}

\begin{theorem}\label{Lmp}
Let $\gamma>0$, $N\geq 1$, $0<b<\min\left\{2, N\right\}$ and $1+\frac{4-2b}{N}<p<{2}^{\circ}$. For any $r>0$ there exists $q_{0}>0$ such that for every $q<q_{0}:$\\
(i) Any minimizing sequence for problem $I^{r}_{q}$ is precompact in $\Sigma(\mathbb{R}^{N})$. \\
(ii) For every $\varphi\in\mathcal{G}_{q}$ there exists a Lagrange multiplier $\omega\in \mathbb{R}$ such that \eqref{Sp} is satisfied with the estimates
\begin{equation*}
-\gamma N<\omega\leq -\gamma N(1-Cq^{\frac{p-1}{2}}).
\end{equation*}
In particular, $\omega\rightarrow -\gamma N$ as $q\rightarrow 0$.\\
(iii) If $u\in \mathcal{G}^{r}_{q}$, then $u(x)=e^{i\theta_{0}}\varphi(x)$, where $\varphi$ is a positive and radially symmetric function and $\theta_{0}\in \mathbb{R}$.
\end{theorem}

We now discuss the orbital stability of standing waves. For $\mathcal{M}\subset \Sigma(\mathbb{R}^{N})$, we say that the set $\mathcal{M}$ is $\Sigma(\mathbb{R}^{N})$-stable under the flow generated by \eqref{GP}  if for all $\epsilon>0$ there exists $\delta>0$ with the following property: if $u_{0}\in \Sigma(\mathbb{R}^{N})$ and $$\inf_{\varphi\in \mathcal{M}}\|u_{0}-\varphi\|_{\Sigma(\mathbb{R}^{N})}<\delta,$$ then the solution $u(t)$ of the Cauchy problem exists for all $t\in \mathbb{R}$ and
\begin{equation*}
\sup_{t\in \mathbb{R}}\inf_{\varphi\in \mathcal{M}}\|u(t)-\varphi\|_{\Sigma(\mathbb{R}^{N})}<\epsilon.
\end{equation*}
Moreover,  we say that the standing wave $e^{i\omega t} \varphi$ is strongly unstable if for each $\epsilon >0$, there exists $u_0 \in \Sigma(\mathbb{R}^{N})$ such that $\|u_0 - \varphi\|_{\Sigma(\mathbb{R}^{N})} <\epsilon$ and the solution $u(t)$ of \eqref{GP} with $u(0)=0$ blows up in finite time.

We have the following stability results for the standing waves of equation \eqref{GP}.
\begin{theorem}\label{Et}
Let $\gamma>0$, $N\geq 1$ and $0<b<\min\left\{2, N\right\}$.\\
(i) If $1<p<1+\frac{4-2b}{N}$, then $\mathcal{G}_{q}$ is $\Sigma(\mathbb{R}^{N})$-stable with respect to  \eqref{GP}.\\
(ii) If $p=1+\frac{4-2b}{N}$ and $q< \|Q\|^{2}_{L^{2}}$, then $\mathcal{G}_{q}$ is $\Sigma(\mathbb{R}^{N})$-stable with respect to  \eqref{GP}.\\
(iii) If $1+\frac{4-2b}{N}<p<{2}^{\circ}$, then for any fixed $r>0$ and $q<q_{0}$ given in the Theorem \ref{Lmp} we
have that the set $\mathcal{G}^{r}_{q}$ is $\Sigma(\mathbb{R}^{N})$-stable with respect to \eqref{GP}.
\end{theorem}

For instability of standing wave solution of \eqref{GP}, we have the following result.
\begin{theorem}\label{INSF}
Let $N\geq 1$, $0<b<\min \{2,N\}$, $\omega>-\gamma N$ and $1+\frac{4-2b}{N} \leq p <2^{\circ}$. Let $\varphi\in \mathcal{M}_{\omega}$. \\
(i) If $d_n\geq d_\omega$, then the standing wave $e^{i\omega t} \varphi$ is strongly unstable in $\Sigma(\mathbb{R}^{N})$.\\
(ii) If $d_n <d_\omega$, then there exists $\delta>0$ and an initial data $u_0$ with $\|u_0 - \varphi\|_{\Sigma}>\delta$ such that the corresponding solution blows up in finite time.
\end{theorem}

This paper is organized as follows. In Section \ref{S:1}, the sharp condition for global existence is established (Theorem \ref{SC}).
In Section \ref{S:2}, we construct and classify finite time blow up solutions at the minimal mass threshold. In Section \ref{S:2/2} we prove the existence of a minimizer for $d_{\omega}$. Section \ref{S:2/3} is devoted to the proof of Theorem \ref{prop blowup}. Section \ref{S:3} contains the proof of Theorems \ref{CMm} and \ref{CC}. In Section \ref{S:4}, we establish the proof of Theorem \ref{Lmp}.
Finally, Theorems \ref{Et} and \ref{INSF}  are proved in Section \ref{S:5}. In Appendix \ref{S:8},  we prove a uniqueness result for  \eqref{Sp}.

{\bf Notation.} 
 The space $L^{2}(\mathbb{R}^{N},\mathbb{C})$  will be denoted  by $L^{2}$ and its norm by $\|\cdot\|_{L^{2}}$. This
space will be equipped with the  real scalar product
\begin{equation*}
\left(u,v\right)_{L^{2}}=\text{Re}\int_{\mathbb{R}^{N}}u\,\overline{v}\,dx\quad u,v\in L^{2}(\mathbb{R}^{N},\mathbb{C}).
\end{equation*}
The space $L^{p}(\mathbb{R}^{N})$, denoted by $L^{p}$ for shorthand, is equipped with the norm $L^{p}$. Throughout this paper, the letter $C$ denotes a constant which may vary from line to line.

\section{The critical mass-case : sharp existence}  
\label{S:1}
The aim of this section is to prove Theorem \ref{SC}. First we observe
\begin{remark}
(i) Let $u\in \Sigma(\mathbb{R}^{N})$. Then the following estimate holds:
\begin{equation}\label{HI}
\int_{\mathbb{R}^{N}}|u|^{2}dx\leq \frac{2}{N}\left(\int_{\mathbb{R}^{N}}|\nabla u|^{2}dx\right)^{\frac{1}{2}}\left(\int_{\mathbb{R}^{N}}|x|^{2}| u|^{2}dx\right)^{\frac{1}{2}}.
\end{equation}
Notice that $2/N$ is the best constant for the inequality \eqref{HI}.\\
(ii) If $Q\in H^{1}(\mathbb{R}^{N})$ satisfies \eqref{Izx}, then the  following identity holds:
\begin{equation}\label{Pi}
\left( \frac{N+2-b}{N}\right)\int_{\mathbb{R}^{N}}|\nabla Q|^{2}dx=\int_{\mathbb{R}^{N}}|x|^{-b}|Q|^{2+\frac{4-2b}{N}}dx.
\end{equation}
\end{remark}
As in \cite{JZ2000}, which deal with the classical case $b=0$, we use the virial identity for the proof of Theorem \ref{SC}. From \eqref{HI}, to show that the $H^{1}(\mathbb{R}^{N})$- norm blow up, it suffices to show that the variance $f(t)$, which is defined by
\begin{equation*}
f(t):=\int_{\mathbb{R}^{N}}|x|^{2}|u(x,t)|^{2}dx
\end{equation*}
vanishes as $t\rightarrow \tau$ for some $\tau<\infty$.

\begin{lemma}\label{L1} 
Let $u$ be a solution of \eqref{GP} on an interval $I=[0, T)$. Then the variance $f$ is the class $C^{2}$ on $I$ and 
satisfies the following identities:
\begin{align*}
f^{\prime}(t)&= 4\mathrm{Im}\int_{\mathbb{R}^{N}}\overline{u}(x,t)(\nabla u(x,t)\cdot x)\,dx, \\ 
f^{\prime\prime}(t)&=16E(u(t))+\frac{4}{p+1}\left(N-Np-2b+4\right)\int_{\mathbb{R}^{N}}|x|^{-b}| u(x,t)|^{p+1}dx-16\gamma^{2}f(t).
\end{align*}
\end{lemma}
This result can be proved along the same lines  as in \cite{VCGG, LGF} and hence omitted. Notice that if $p=1+\frac{4-2b}{N}$ in the previous lemma, then $f^{\prime\prime}(t)=16E(u_{0})-16\gamma^{2}f(t)$. Throughout the rest of this section we assume that $p=1+\frac{4-2b}{N}$.
\begin{lemma}\label{L2}
Let $u_{0}\neq0$ be such that $f(0)\geq 2\gamma^{-2}E(u_{0})$.  Then the solution $u$ of \eqref{GP} corresponding to $u_{0}$ blows up in finite time.
\end{lemma}
\begin{proof}Since $f^{\prime\prime}(t)=16E(u_{0})-16\gamma^{2}f(t)$, a straightforward calculation gives first 
\begin{equation}\label{Pii}
f(t)=r\mathrm{sin}(4\gamma t+\theta)+\gamma^{-2}E(u_{0}),
\end{equation}
where $r\geq 0$ and $\theta\in [0, 2\pi)$ are constants determined by $f(0)$ and $f^{\prime}(0)$.
We also have 
\begin{equation}\label{P2}
r^{2}=[f(0)-\gamma^{-2}E(u_{0})]^{2}+\frac{\gamma^{-2}}{16}[f^{\prime}(0)]^{2}.
\end{equation}
Since $f(0)\geq 2\gamma^{-2}E(u_{0})$, it follows that $r\geq \gamma^{-2}E(u_{0})$. Thus from \eqref{Pii} and \eqref{P2}, we see that there exists $\tau<\infty$ such that 
\begin{equation*}
\lim_{t\rightarrow\tau}f(t)=0.
\end{equation*}
Inequality \eqref{HI} implies that $\lim_{t\rightarrow\tau}\|\nabla u(t)\|^{2}_{L^{2}}=+\infty$. This shows that $u(x,t)$ blows up in finite time, which completes the proof of lemma.
\end{proof}

Now we give the proof of Theorem \ref{SC}.
\begin{proof}[ \bf {Proof of Theorem \ref{SC}}] 
First, as noted in the introduction, we have that for every $u\in H^{1}(\mathbb{R}^{N})$,
\begin{equation}\label{Lk}
\int_{\mathbb{R}^{N}}|x|^{-b}|u|^{\frac{4-2b}{N}+2}dx\leq {\|Q\|^{-\left(\frac{4-2b}{N}\right)}_{L^{2}}}\left(\frac{2+N-b}{N}\right)\|\nabla u\|^{2}_{L^{2}}\|u\|^{\frac{4-2b}{N}}_{L^{2}}.
\end{equation}
Notice that ${\|Q\|^{-\left(\frac{4-2b}{N}\right)}_{L^{2}}}\left(\frac{2+N-b}{N}\right)$ is the best constant for the above inequality. 
Consider a local solution $u\in C([0,T), \Sigma(\mathbb{R}^{N}))$ of the Cauchy problem of \eqref{GP}, as given by Proposition \ref{WP}, where $[0,T)$ is the maximal existence time. In view of \eqref{Lk} and the conservation of charge and energy, it is clear that 
\begin{equation*}
\frac{1}{2}\|\nabla u(t)\|^{2}_{L^{2}}\left(1-\left(\frac{\|u_{0}\|_{L^{2}}}{\|Q\|_{L^{2}}}\right)^{\frac{4-2b}{N}}\right)+\frac{\gamma^{2}}{2}\int_{\mathbb{R}^{N}}|x|^{2}| u(x, t)|^{2}dx\leq E(u_{0}).
\end{equation*}
Now, since $\|u_{0}\|_{L^{2}}<\|Q\|_{L^{2}}$, it follows that $\|\nabla u(t)\|^{2}_{L^{2}}$ is bounded for all $t\in [0,T)$. From Proposition \ref{WP} it yields that $u(x,t)$ globally exists in $t\in [0,+\infty)$, which completes the proof of Item (i). 

On the other hand, for $\lambda>0$ and $c\in \mathbb{C}$, $|c|\geq 1$, we take the initial date $u(0,t)=c\lambda^{\frac{N}{2}}Q(\lambda x)$. Clearly  $\|u_{0}\|_{L^{2}}=|c|\|Q\|_{L^{2}}\geq \|Q\|_{L^{2}}$. Now combining \eqref{Es} and \eqref{Pi}, it follows from straightforward calculations that 
\begin{equation*}
2E(u_{0})=\|\nabla Q\|^{2}_{L^{2}}|c|^{2}\lambda^{2}\left(1-|c|^{\frac{4-2b}{N}}\right)+\gamma^{2}f(0)\leq \gamma^{2}f(0).
\end{equation*}
Therefore, by Lemma \ref{L2} we have that $u(x,t)$ blows up in finite time, and this finishes the proof of theorem.
\end{proof}

\section{Classification of minimal mass blow up solutions}
\label{S:2}
In this section, we give the proof of Theorem \ref{CMm23}. For any function $u: \mathbb{R}^{N}\times I\rightarrow \mathbb{C}$, we define  
\begin{equation}\label{ul}
u_{L}(x,t)=\left(\frac{1}{\mathrm{cos}\,2\gamma t}\right)^{\frac{N}{2}}e^{-i\frac{\gamma}{2}|x|^{2}\mathrm{tan}\,\,2\gamma t} u\left(\frac{x}{\mathrm{cos}\,2\gamma t}, \frac{1}{2\gamma}\mathrm{tan}\,2\gamma t\right).
\end{equation}
Notice that $u_{L}$ is defined on the time interval $\mathrm{tan}^{-1}(I):=\left\{\mathrm{tan}^{-1}(t), t\in I\right\}$ and $u_{L}(x,0)=u(x,0)$; for more details we refer to \cite{TaoC, RCIH}. We first prove a key lemma to obtain Theorem \ref{CMm23}.

\begin{lemma}\label{Lbn}
Let $\gamma>0$.\\
(i) Assume that $u$ is a solution of the free (i.e.,zero-potential) inhomogeneous nonlinear Schr\"odinger equation 
\begin{equation}\label{Lp}
i\partial_{t}u+\Delta u+ |x|^{-b}|u|^{p-1}u=0,
\end{equation}
on a interval $I$. Then the function $u_{L}$ defined in \eqref{ul} solves the inhomogeneous nonlinear Schr\"odinger equation with attractive harmonic potential
\begin{equation*}
i\partial_{t}u_{L}+\Delta u_{L}- \gamma^{2}|x|^{2}u_{L}+|\mathrm{cos}\, 2\gamma t|^{\frac{N}{2}(p-1)-2+b}|x|^{-b}|u_{L}|^{p-1}u_{L}=0,
\end{equation*}
with $\|u_{L}\|_{L^{2}}=\|u\|_{L^{2}}$. In particular,  if $p=1+\frac{4-2b}{N}$, then $u_{L}$ is a solution of \eqref{GP} on $\mathrm{tan}^{-1}(I)$. \\
(ii) Reciprocally, assume that $u\in \Sigma(\mathbb{R}^{N})$ is a solution of \eqref{GP} with $p=1+\frac{4-2b}{N}$, then the function $u_{L^{-1}}$, defined by 
\begin{equation}\label{uli}
u_{L^{-1}}(x,t)=\frac{1}{(1+4(\gamma t)^{2})^{\frac{N}{2}}}e^{i\frac{4\gamma^{2}t}{1+4(\gamma t)^{2}}{|x|^{2}}}u\left(\frac{x}{\sqrt{1+4\gamma t^{2}}}, \frac{1}{2\gamma}\mathrm{tan}^{-1}\,2\gamma t\right)
\end{equation}
solves \eqref{Lp} with $p=1+\frac{4-2b}{N}$.
\end{lemma}
\begin{proof}
For simplicity, we assume that $\gamma=\frac{1}{2}$.  We can easily check that
\begin{align*}
\partial_{t}u_{L}(x,t)=\frac{e^{-i\frac{1}{4}|x|^{2}\mathrm{tan}\, t}}{\mathrm{cos}^{\frac{N}{2}+2}\, t}\left[\partial_{t}u -i\frac{|x|^{2}}{4}u+\mathrm{sin}\, t\,\, \nabla u \cdot x + \frac{N}{2}\mathrm{sin}\, t\,\,\mathrm{cos}\, t\,\, u\right]\left(\frac{x}{\mathrm{cos}\, t}, \mathrm{tan}\, t \right)
\end{align*}
and 
\begin{multline*}
\Delta u_{L}(x,t)=\frac{e^{-i\frac{1}{4}|x|^{2}\mathrm{tan}\, t}}{\mathrm{cos}^{\frac{N}{2}+2}\, t}\left[-\frac{|x|^{2}}{4}\mathrm{sin}^{2}\, t\, u-i\,\mathrm{sin}\, t\, \nabla u \cdot x -\frac{i}{2}N\mathrm{sin}\, t\,\mathrm{cos}\, t\, u+\Delta u\right]\left(\frac{x}{\mathrm{cos}\, t}, \mathrm{tan}\, t \right).
\end{multline*}
Thus, we see that 
\begin{align*}
i\, \partial_{t}u_{L}(x,t)+\Delta u_{L}(x,t)&=\frac{e^{-i\frac{1}{4}|x|^{2}\mathrm{tan}\, t}}{\mathrm{cos}^{\frac{N}{2}+2}\, t}\left[\frac{|x|^{2}}{4}\mathrm{cos}^{2}\,t\,u + i\,\partial_{t}u +\Delta u \right]\left(\frac{x}{\mathrm{cos}\, t}, \mathrm{tan}\, t \right)\\
&=\frac{e^{-i\frac{1}{4}|x|^{2}\mathrm{tan}\, t}}{\mathrm{cos}^{\frac{N}{2}+2}\, t}\left[\frac{|x|^{2}}{4}\mathrm{cos}^{2}\,t\,u -|x|^{-b}|u|^{p-1}u \right]\left(\frac{x}{\mathrm{cos}\, t}, \mathrm{tan}\, t \right)\\
&=\frac{|x|^{2}}{4}u_{L}(x,t)-|\mathrm{cos}\, t|^{\frac{N}{2}(p-1)-2+b}|x|^{-b}|u_{L}(x,t)|^{p-1}u_{L}(x,t).
\end{align*}
This proves the first statement of lemma. Similarly, the second statement of the lemma follows from a straightforward calculation. With this the lemma is proved
\end{proof}
It is important to note that the transforms \eqref{ul} and \eqref{uli} do not alter the initial data $u_{0}$; notice also that $\|u_{L}\|_{L^{2}}=\|u_{L^{-1}}\|_{L^{2}}=\|u\|_{L^{2}}$.  

Theorem \ref{CMm23} follows from the previous lemma and from the following result of Combet and Genoud \cite[Theorem 1]{VCGG}. 
\begin{proposition}\label{Ph}
Let $u_{0}\in H^{1}(\mathbb{R}^{N})$ with $\|u_{0}\|_{L^{2}}=\|Q\|_{L^{2}}$, where $Q$ is defined by \eqref{Izx}. Assume that the solution  $u$  of \eqref{Lp} blows up in finite time $T>0$. Then there exist $\theta_{0}\in \mathbb{R}$ and $\lambda>0$ such that 
\begin{equation*}
u(t)=S_{\lambda, \theta_{0}}(T-t) \quad \text{for every $t\in [0, T)$},
\end{equation*}
where $S_{\lambda, \theta_{0}}$ is defined by \eqref{AX}.
\end{proposition}

Now we give the proof of Theorem \ref{CMm23}.
\begin{proof}[ \bf {Proof of Theorem \ref{CMm23}}]  Let $p=1+\frac{4-2b}{N}$. Assume that $u$ is a solution of the Cauchy problem \eqref{GP} such that $\|u_{0}\|_{L^{2}}=\|Q\|_{L^{2}}$ and $\lim_{t\rightarrow T}\|\nabla u(t)\|^{2}_{L^{2}}=+\infty$ with $T<\frac{\pi}{4\gamma}$. We set $v(x,t):=u_{L^{-1}}(x,t)$. From Lemma \ref{Lbn} (ii), we have that $v(x,t)$ is a solution of \eqref{Lp} with $v(x,t)=u_{0}$, which blows up in finite time $T^{\ast}:=\frac{1}{2\gamma} \mathrm{tan}\, (2\gamma T)$. By Proposition \ref{Ph} we know that there exist $\theta_{0}\in \mathbb{R}$ and $\lambda_{0}>0$ such that 
\begin{equation*}
v(t)=S_{\lambda, \theta_{0}}(x, T^{\ast}-t) \quad \text{for every $t\in [0, T^{\ast})$},
\end{equation*}
whence, again by Lemma \ref{Lbn} and from uniqueness result of Proposition \ref{WP}, it follows that
\begin{equation*}
u(t)=v_{L}(t)=\left(\frac{1}{\mathrm{cos}\,2\gamma t}\right)^{\frac{N}{2}}e^{-i\frac{\gamma}{2}|x|^{2}\mathrm{tan}\,\,2\gamma t}S_{\lambda, \theta_{0}}\left(\frac{x}{\mathrm{cos}\, 2\gamma t}, T^{\ast}-\frac{1}{2\gamma} \mathrm{tan}\, 2\gamma t\right).
\end{equation*}
Finally, since 
\begin{equation*}
T^{\ast}-\frac{1}{2\gamma} \mathrm{tan}\, (2\gamma t)=\frac{\mathrm{sin}\,2\gamma(T-t)}{2\gamma\mathrm{cos}\,2\gamma T\,\, \mathrm{cos}\,2\gamma t},
\end{equation*}
we see that
\begin{equation*}
u(t)=\left(\frac{1}{\mathrm{cos}\,2\gamma t}\right)^{\frac{N}{2}}e^{-i\frac{\gamma}{2}|x|^{2}\mathrm{tan}\,2\gamma t}S_{{\lambda_{0}}, \theta_{0}}\left(\frac{x}{\mathrm{cos}\,2\gamma t}, \frac{\mathrm{sin}\,2\gamma(T-t)}{2\gamma\mathrm{cos}\,2\gamma T\,\, \mathrm{cos}\,2\gamma t}\right),\end{equation*}
which completes of proof.
\end{proof}

\section{Existence of minimizers}
\label{S:2/2}

The aim this section is to prove Theorem \ref{prop GS}.
\begin{proof}[ \bf {Proof of Theorem \ref{prop GS}}] 
	Let $u \in \Sigma \backslash \{0\}$ be such that $K_\omega(u)=0$. We have $\|u\|^2_{H_\omega} = P(u)$. Using the Gagliardo-Nirenberg inequality
	\[
	P(u) \lesssim \|\nabla u\|_{L^2}^{\frac{N(p-1)}{2}+b} \|u\|_{L^2}^{p+1-\frac{N(p-1)}{2}-b}
	\]
	together with the Young's inequality, we have
	\[
	P(u) \leq C_1 (\|\nabla u\|^2_{L^2} + \|u\|^2_{L^2})^{\frac{p+1}{2}} \leq C_2 \|u\|_{H_\omega}^{p+1} = C_2 P(u)^{\frac{p+1}{2}}.
	\]
	This implies that
	\[
	P(u) >\left( \frac{1}{C_2} \right)^{\frac{2}{p-1}} >0.
	\]
	On the other hand,
	\[
	S_\omega(u) = \frac{1}{2} \|u\|^2_{H_\omega} -\frac{1}{p+1} P(u) =\frac{p-1}{2(p+1)}P(u) > \frac{p-1}{2(p+1)} \left( \frac{1}{C_2} \right)^{\frac{2}{p-1}} >0.
	\]
	Taking the infimum, we obtain $d_\omega>0$.
	
	Let $(u_n)_{n\geq 1}$ be a minimizing sequence of $d_\omega$. Since $K_\omega(u_n)=0$, we have $\|u_n\|^2_{H_\omega} = P(u_n)$ for all $n\geq 1$. Thus, 
	\[
	S_n(u_n) = \frac{p-1}{2(p+1)} \|u_n\|^2_{H_\omega} \rightarrow d_\omega \quad \text{as } n \rightarrow \infty.
	\]
	We infer that there exists a constant $C>0$ such that $\|u_n\|^2_{H_\omega} \leq \frac{2(p+1)}{p-1} d_\omega + C$ for all $n\geq 1$. For $\gamma>0$ and $\omega>-\gamma N$ fixed, $\|u\|^2_{H_\omega} \sim \|u\|^2_{\Sigma}$. This implies that the sequence $(u_n)_{n\geq 1}$ is a bounded in $\Sigma$. There exists $u_0 \in \Sigma$ such that up to a subsequence, we can suppose that $v_n \rightharpoonup u_0$ weakly in $\Sigma$. Since $\Sigma \rightarrow L^{r+1}(\mathbb{R}^N)$ compact (see \cite[Lemma 3.1]{JZ2000}) for $1\leq r<1+\frac{4}{N-2}$ if $N\geq 3$ and $1\leq r<\infty$ if $N=1,2$. This implies that $u_n \rightarrow u_0$ strongly in $L^{r+1}$ with $r$ as above. We now show that $u_0$ is a minimizer of $d_\omega$. Since $u_n \rightharpoonup u_0$ weakly in $\Sigma$, we have
	\[
	\|u_0\|^2_{H_\omega} \leq \liminf_{n\rightarrow \infty} \|u_n\|^2_{H_\omega}.
	\]
	We now claim that for $N\geq 1, 0<b<\min\{2,N\}$ and $1<p<2^{\circ}$,
\begin{equation}
\int_{\mathbb{R}^{N}}|x|^{-b}|u_{n}|^{p+1}dx\rightarrow \int_{\mathbb{R}^{N}}|x|^{-b}|u|^{p+1}dx\quad \text{as $n\rightarrow \infty$.}  \label{conver}
\end{equation}
We have 
\begin{multline*}
\left|\int_{\mathbb{R}^{N}}|x|^{-b}|u_{n}|^{p+1}dx - \int_{\mathbb{R}^{N}}|x|^{-b}|u|^{p+1}dx \right| \leq \||x|^{-b} (|u_n|^{p+1} - |u_0|^{p+1})\|_{L^1} \\
\leq \||x|^{-b} (|u_n|^{p+1} - |u_0|^{p+1})\|_{L^1(B)} + \||x|^{-b} (|u_n|^{p+1} - |u_0|^{p+1})\|_{L^1(B^c)},
\end{multline*}
where $B$ is the unit ball in $\mathbb{R}^N$ and $B^c=\mathbb{R}^N \backslash B$. 

On $B$, we bound
\[
\||x|^{-b} (|u_n|^{p+1} - |u_0|^{p+1})\|_{L^1(B)} \lesssim \||x|^{-b}\|_{L^\delta(B)} \||u_n|^{p+1} - |u_0|^{p+1}\|_{L^\mu}
\]
provided $\delta, \mu \geq 1, 1=\frac{1}{\delta} + \frac{1}{\mu}$. The term $\||x|^{-b}\|_{L^\delta(B)}$ is finite provided $\frac{N}{\delta}>b$. Thus, $\frac{1}{\delta} >\frac{b}{N}$, and $\frac{1}{\mu}=1-\frac{1}{\delta} <\frac{N-b}{N}$. We next bound
\[
\||u_n|^{p+1} - |u_0|^{p+1}\|_{L^\mu} \lesssim (\|u_n\|^p_{L^\tau} + \|u_0\|^p_{L^\tau}) \|u_n-u_0\|_{L^\sigma}
\]
provided
\begin{align}
\frac{p}{\tau}+\frac{1}{\sigma} = \frac{1}{\mu} < \frac{N-b}{N}. \label{in ball}
\end{align}
Using the embedding $\Sigma \rightarrow L^{r+1}(\mathbb{R}^N)$ for $1\leq r<1+\frac{4}{N-2}$ if $N\geq 3$ and $1\leq r<\infty$ if $N=1,2$, we are able to choose $\tau \in \left[2,\frac{2N}{N-2}\right)$ if $N\geq 3$ and $\tau \in [2,\infty)$ if $N=1,2$ so that $\|u_n\|_{L^\tau} \lesssim \|u_n\|_{\Sigma}$ (similarly for $u_0$). In the case $N\geq 3$, we have
\[
\frac{p(N-2)}{2N} + \frac{1}{\sigma} <\frac{N-b}{N}.
\]
Since $u_n \rightarrow u_0$ in $L^{r+1}$ with $1 \leq r <1 +\frac{2}{N-2}$, it follows that
\[
\frac{p(N-2)}{2N} + \frac{N-2}{2N} < \frac{N-b}{N}. 
\]
This condition is satisfied since $p<1 +\frac{4-2b}{N-2}$. Since $u_n \rightarrow u_0$ in $L^{r+1}$ with $1 \leq r <\infty$, we are able to choose $\tau$ and $\sigma$ large enough so that \eqref{in ball} holds. As a consequence, we prove
\[
\||x|^{-b} (|u_n|^{p+1} - |u_0|^{p+1})\|_{L^1(B)} \quad \text{as } n \rightarrow \infty.
\]
On $B^c$, we bound
\begin{align*}
\||x|^{-b}(|u_n|^{p+1} - |u_0|^{p+1})\|_{L^1(B^c)} &\leq \||u_n|^{p+1} - |u_0|^{p+1}\|_{L^1} \\ &\lesssim \|u_n-u_0\|_{L^{p+1}} (\|u_n|^p_{L^{p+1}} + \|u_0\|^p_{L^{p+1}}) \\
&\lesssim \|u_n-u_0\|_{L^{p+1}} (\|u_n\|^p_{\Sigma} + \|u_0\|^p_{\Sigma}) \rightarrow 0.
\end{align*}
Combining two terms, we prove the claim.
	
	 Thus
	\[
	K_\omega(u_0) \leq \liminf_{n\rightarrow \infty} K_\omega(u_n)=0.
	\]
 	We also have
	\[
	\lim_{n\rightarrow \infty} S_\omega(u_n) = \lim_{n\rightarrow \infty} \frac{p-1}{2(p+1)} P(u_n) = \frac{p-1}{2(p+1)} P(u_0) = d_\omega.
	\]
	Suppose $K_\omega(u_0) <0$, thus $\|u_0\|^2_{H_\omega} <P(u_0)$. 
	We have
	\[
	K_\omega(\lambda u_0) = \lambda^2 \|u_0\|^2_{H_\omega} - \lambda^{p+1} P(u_0).
	\]
	This implies that $K_\omega(\lambda_0 u_0) =0$, where
	\[
	\lambda_0 = \left(\frac{\|u_0\|^2_{H_\omega}}{P(u_0)}\right)^{\frac{1}{p-1}} \in (0,1).
	\]
	By definition of $d_\omega$,
	\[
	d_\omega \leq S_\omega(\lambda_0 u_0) = \frac{p-1}{2(p+1)} P(\lambda_0 u_0) = \frac{p-1}{2(p+1)} \lambda_0^{p+1} P(u_0) < \frac{p-1}{2(p+1)} P(u_0) = d_\omega.
	\]
	This is a contradiction. Therefore, $K_\omega(u_0) =0$. This combined with the fact $S_\omega(u_0) = \frac{p-1}{2(p+1)} P(u_0) = d_\omega$ imply that $u_0$ is a minimizer of $d_\omega$. 
	
	It remains to show $u_0$ solves the elliptic equation \eqref{Sp}. Since $u_0$ is a minimizer of $d_\omega$, there exists a Lagrange multiplier $\mu \in \mathbb{R}$ such that $S'_\omega(u_0) = \mu K'_\omega(u_0)$. We have
	\begin{align}
	0= K_\omega(u_0) = \langle S'_\omega(u_0), u_0\rangle = \mu \langle K'_\omega(u_0), u_0\rangle. \label{lagr mu}
	\end{align}
	On the other hand,
	\[
	K'_\omega(u_0) = 2(-\Delta) u_0 +2\gamma |x|^2 u_0+2\omega u_0 - (p+1) |x|^{-b} |u_0|^{p-1} u_0.
	\]
	Thus,
	\[
	\langle K'_\omega(u_0), u_0\rangle = 2\|u_0\|^2_{H_\omega} - (p+1) P(u_0) = -(p-1) P(u_0) <0.
	\]
	This together with \eqref{lagr mu} imply that $\mu=0$. So, $S'_\omega(u_0)=0$ or $u_0$ is a solution of \eqref{Sp}. This proves the first part of the statement.
Now	let $u$ be a complex valued minimizer for $d_\omega$. We claim that there exists $\theta\in \mathbb{R}$ such that
	$u(x)=e^{i\theta}\varphi(x)$, where $\varphi$ is a positive real valued minimizer. Indeed, since $\|\nabla (|u|)\|^{2}_{L^{2}}\leq \|\nabla u\|^{2}_{L^{2}}$, it is clear that  $S_\omega(|u|)\leq S_\omega(u)$ and $K_\omega(|u|)\leq K_\omega(u)=0$. In particular, $|u|\in\mathcal{M}_{\omega}$ and 
	\begin{equation}\label{Gca}
	\|\nabla |u|\|^{2}_{L^{2}}= \|\nabla u\|^{2}_{L^{2}}.
	\end{equation}
	From the Euler-Lagrange equation \eqref{Sp} and an elliptic regularity regularity/bootstrap argument we see that $u\in C^{1}(\mathbb{R}^{N}, \mathbb{C})$ (see \cite[Sections 2.1 and 2.2]{FGCS} and \cite{ABRF}). Moreover, the positivity of $|u|$ follows from the maximum principle and thus $u\in C^{1}(\mathbb{R}^{N}, \mathbb{C}\setminus\left\{0\right\})$. 
	
	We set $w(x):=\frac{u(x)}{|u(x)|}$. Since $|w|^{2}=1$, it follows that $\text{Re}(\overline{w}\,\nabla w)=0$ and 
	\begin{equation*}
	\nabla u=(\nabla |u|)w+|u|\nabla w=w(\nabla |u|+|u|\overline{w}\nabla w).
	\end{equation*}
	Therefore, we see that $|\nabla u|^{2}=|\nabla |u||^{2}+|u|^{2}|\nabla w|^{2}$.  From \eqref{Gca} we get 
	\begin{equation*}
	\int_{\mathbb{R}^{N}}|u|^{2}|\nabla w|^{2}dx=0,
	\end{equation*}
	and thus $|\nabla w|=0$. Hence $w$ is constant with $|w|=1$, we infer that there exists $\theta\in \mathbb{R}$ such that $u=e^{i \theta}\varphi(x)$ where $\varphi(x):=|u(x)|$.  This prove the claim. We now prove that $\varphi$ is necessarily radial and radially decreasing. Indeed, denoting by $\varphi^{\ast}$ the Schwarz rearrangement  of $\varphi$, it is well known that (see \cite{HHa})
	\begin{align}\label{S1}
	\int_{\mathbb{R}^{N}}|x|^{2}|\varphi^{\ast}(x)|^{2}\,dx&<\int_{\mathbb{R}^{N}}|x|^{2}\varphi^{2}(x)\,dx \quad \text{unless $\varphi=\varphi^{\ast}$,}\\\label{S2}
	\int_{\mathbb{R}^{N}}|x|^{-b}|\varphi^{\ast}(x)|^{p+1}\,dx&>\int_{\mathbb{R}^{N}}|x|^{-b}\varphi^{p+1}(x)\,dx \quad \text{unless $\varphi=\varphi^{\ast}$.}
	\end{align}
	Thus, from $\|\nabla \varphi^{\ast}\|^{2}_{L^{2}}\leq \|\nabla \varphi\|^{2}_{L^{2}}$, we infer that if $\varphi$ is not radial, then 
	$S_\omega(\varphi^{\ast})<S_\omega(\varphi)=d_{\omega}$ and $K_\omega(\varphi^{\ast})<K_\omega(\varphi)=0$, a contradiction.
	This prove that $\varphi$ is  radial and radially decreasing. 
\end{proof}

\section{Sharp thresholds for blowup and global existence in the mass-critical and mass-supercritical cases}
\label{S:2/3}
This section is devoted to the proof of Theorem \ref{prop blowup}. We have divided the proof into a sequence of lemmas.

\begin{lemma} \label{lem poho}
	Let $\gamma>0, N \geq 1, 0<b<\min\{2,N\}, 1<p<2^{\circ}$ and $\omega>-\gamma N$. There exists $u \in \Sigma \backslash \{0\}$ such that $K_\omega(u) = I(u) =0$.
\end{lemma}
\begin{proof}
	By Proposition \ref{prop GS}, there exists a non-trivial solution $u$ to the elliptic equation \eqref{Sp}. Multiplying both sides of \eqref{Sp} with $\overline{u}$ and integrating over $\mathbb{R}^N$, we have
	\begin{align}
	\|\nabla u\|^2_{L^2} + \omega\|u\|^2_{L^2} + \gamma^2 \|xu\|^2_{L^2} - P(u)=0. \label{id 1}
	\end{align}
	On the other hand, multiplying both sides of \eqref{Sp} with $x \cdot \nabla \overline{u}$, integrating over $\mathbb{R}^N$ and taking the real part, we have
	\begin{align}
	\frac{2-N}{2}\|\nabla u\|^2_{L^2} -\frac{N\omega}{2} \|u\|^2_{L^2} -\frac{N+2}{2} \gamma^2 \|xu\|^2_{L^2} + \frac{N-b}{p+1} P(u)=0.\label{id 2}
	\end{align}
	By \eqref{id 1}, it is obvious that $K_\omega(u)=0$. Multiplying both sides of \eqref{id 1} with $\frac{d}{2}$ and adding to \eqref{id 2}, we get
	\[
	\|\nabla u\|^2_{L^2} - \gamma^2 \|xu\|^2_{L^2} - \frac{N(p-1)+2b}{2(p+1)} P(u)=0,
	\]
	which implies that $I(u)=0$.
\end{proof}

\begin{lemma}
	Let $\gamma>0, N\geq 1, 0<b<\min\{2,N\}, 1+\frac{4-2b}{N} \leq N <2^{\circ}$ and $\omega>-\gamma N$. Then the set $\mathcal{N}$ is not empty.
\end{lemma}
\begin{proof}
	By Lemma \ref{lem poho}, there exists $u \in \Sigma \backslash \{0\}$ such that $K_\omega=I(u)=0$. Set $u^\lambda(x) =\lambda u(x)$. We have
	\begin{align*}
	K_\omega(u^\lambda) &= \lambda^2 \|u\|^2_{H_\omega} - \lambda^{p+1} P(u), \\
	I(u^\lambda)&= \lambda^2 (\|\nabla u\|^2_{L^2} - \gamma^2\|xu\|^2_{L^2}) - \frac{N(p-1)+2b}{2(p+1)}\lambda^{p+1} P(u).
	\end{align*}
	Since $K_\omega(u)=I(u)=0$, the equations $K_\omega(u^\lambda)=0$ and $I(u^\lambda)=0$ admit unique non-zero solution $\lambda=1$. Therefore, $K_\omega(u^\lambda)<0, I(u^\lambda)<0$ for all $\lambda>1$. Consider 
	\[
	A(\lambda):= \|\nabla u^\lambda\|^2_{L^2} - \frac{N(p-1)+2b}{2(p+1)} P(u^\lambda).
	\]
	Since $I(u)=0$, we have $A(1)>0$. By continuity, there exists $\lambda_0>1$ such that $A(\lambda_0)>0$. We denote $v(x) = u^{\lambda_0}(x)$. Set
	\[
	v_\mu(x) := \mu^{\frac{2-b}{p-1}} v(\mu x), \quad \mu>0.
	\]
	A calculation shows that
	\begin{align*}
	K_\omega(v_\mu) &= \mu^a (\|\nabla v\|^2_{L^2} - P(v)) + \mu^{a-4}\gamma^2 \|xu\|^2_{L^2} + \mu^{a-2} \omega \|u\|^2_{L^2}, \\
	I(v_\mu) &= \mu^a \left(\|\nabla v\|^2_{L^2} - \frac{N(p-1)+2b}{2(p+1)} P(v)\right) - \mu^{a-4} \gamma^2 \|xu\|^2_{L^2},
	\end{align*}
	where
	\[
	a= \frac{4-2b-(N-2)(p-1)}{p-1} >0.
	\]
	Since $I(v)<0$ and $\lim_{\mu \rightarrow +\infty} I(v_\mu) = +\infty$, there exists $\mu_0>1$ such that $I(v_{\mu_0})=0$. On the other hand, $K_\omega(v)<0$ implies that $\|\nabla v\|^2_{L^2} - P(v) <0$. Moreover, since $a>0, a-2\leq 0$ and $a-4<0$, we see that $K_\omega(v_\mu) <0$ for all $\mu>1$. We obtain $v_{\mu_0} \in \Sigma \backslash \{0\}$, $K_\omega(v_{\mu_0})<0$ and $I(v_{\mu_0})=0$ or $v_{\mu_0} \in \mathcal{N}$. 
\end{proof}

\begin{lemma} \label{lem dn}
	Let $\gamma>0, N\geq 1, 0<b<\min\{2,N\}, 1+ \frac{4-2b}{N} \leq N <2^{\circ}$ and $\omega>-\gamma N$. Then $d_n>0$.
\end{lemma}
\begin{proof}
	Let $u \in \Sigma \backslash \{0\}$ be such that $K_\omega(u)<0$ and $I(u)=0$. Since $I(u)=0$, we have
	\[
	\|\nabla u\|^2_{L^2} - \gamma^2 \|xu\|^2_{L^2} = \frac{N(p-1)+2b}{2(p+1)} P(u).
	\]
	Thus,
	\begin{multline}
	S_\omega(u) = \left(\frac{1}{2} - \frac{2}{N(p-1)+2b}\right) \|\nabla u\|^2_{L^2}
	+ \left(\frac{1}{2}	+ \frac{2}{N(p-1)+2b}\right) \gamma^2 \|xu\|^2_{L^2}+\frac{\omega}{2} \|u\|^2_{L^2}. \label{S_ome u}
	\end{multline}
	We now consider two cases: $L^2$-supercritical case and $L^2$-critical case. 
	
	Case 1: $L^2$-supercritical case $1 + \frac{4-2b}{N} < p <2^{\circ}$. By the Gagliardo-Nirenberg inequality, we have
	\begin{align*}
	P(u) & \lesssim \|\nabla u\|^{\frac{N(p-1)}{2}+b}_{L^2} \|u\|^{p+1-\frac{N(p-1)}{2}-b}_{L^2} \\
	&\leq C_1 (\|\nabla u\|^2_{L^2} + \|u\|^2_{L^2})^{\frac{p+1}{2}}\\
	&\leq C_2 \|u\|^{p+1}_{H_\omega}.
	\end{align*}
	Since $K_\omega(u)<0$, it follows that $\|u\|^2_{H_\omega} < P(u)$. Thus, $\|u\|^2_{H_\omega}< P(u)  \leq C_2 \|u\|^{p+1}_{H_\omega}$. We get
	\[
	\|u\|^2_{H_\omega}>\left(\frac{1}{C_2}\right)^{\frac{2}{p-1}}>0. 
	\]
	On the other hand, by \eqref{S_ome u} and the fact $\frac{1}{2}-\frac{2}{N(p-1)+2b}>0$ in this case, we have
	\[
	S_\omega(u) \geq C_3 \|u\|^2_{H_\omega} > C_3 \left(\frac{1}{C_2}\right)^{\frac{2}{p-1}}>0.
	\]
	Taking the infimum, we obtain $d_n>0$.
	
	Case 2: $L^2$-critical case $p=1+\frac{4-2b}{N}$. Assume $d_n=0$, there exists $(u_n)_{n\geq 1} \subset \Sigma \backslash \{0\}$, $K_\omega(u_n) <0$ and $I(u_n)=0$ for all $n\geq 1$ and $S_\omega(u_n) \rightarrow 0$ as $n\rightarrow \infty$. It follows from \eqref{S_ome u} that 
	\begin{align}
	\|u_n\|^2_{L^2} \rightarrow 0, \quad \|xu_n\|^2_{L^2} \rightarrow 0 \quad \text{as } n \rightarrow \infty. \label{limit}
	\end{align}
	Since $K_\omega(u_n) <0$, the sharp Gagliardo-Nirenberg inequality implies that
	\begin{align}
	\|u_n\|^2_{H_\omega} < P(u_n) \leq C \|\nabla u_n\|^2_{L^2} \|u_n\|^{\frac{4-2b}{N}}_{L^2}. \label{ineq 1}
	\end{align}
	For the constant $C$ in \eqref{ineq 1}, we have from \eqref{limit} that for $n$ sufficiently large,
	\[
	\|\nabla u_n\|^2_{L^2} > C \|\nabla u_n\|^2_{L^2} \|u_n\|^{\frac{4-2b}{N}}_{L^2}.
	\]
	It follows that 
	\begin{align}
	\|u_n\|^2_{H_\omega} > C \|\nabla u_n\|^2_{L^2} \|u_n\|^{\frac{4-2b}{N}}_{L^2}. \label{ineq 2}
	\end{align}
	The inequalities \eqref{ineq 1} and \eqref{ineq 2} contradict each other. Therefore, $d_n>0$.
\end{proof}

\begin{lemma}\label{Povs}
	Let $\gamma>0, N\geq 1, 0<b<\min\{2,N\}, 1+\frac{4-2b}{N} \leq p <2^{\circ}$ and $\omega>-\gamma N$. Then $d>0$.
\end{lemma}
\begin{proof}
	It comes from Theorem $\ref{prop GS}$ and Lemma $\ref{lem dn}$.
\end{proof}

\begin{lemma} \label{lem inva}
	Let $\gamma>0, N\geq 1, 0<b<\min\{2,N\}, 1+\frac{4-2b}{N} \leq p <2^{\circ}$ and $\omega>-\gamma N$. Then the sets $K_\pm, R_\pm$ are invariant under the flow of \eqref{GP}.
\end{lemma}
\begin{proof}
	We only give the proof for $K_-$, the ones for $K_+, R_\pm$ are similar. Let $u_0 \in K_-$, i.e. $S_\omega(u_0)<d$, $K_\omega(u_0)<0, I(u_0)<0$. By conservation of mass and energy,
	\begin{align}
	S_\omega(u(t)) = S_\omega(u_0) <d, \quad \forall t \in [0,T). \label{global}
	\end{align}
	We now prove $K_\omega(u(t))<0$ for all $t\in [0,T)$. Suppose there exists $t_0>0$ such that $K_\omega(u(t_0)) \geq 0$. By the continuity of $t \mapsto K_\omega(u(t))$, there exists $t_1 \in (0,t_0]$ such that $K_\omega(u(t_1)) =0$. By the definition of $d_\omega$, $S_\omega(u(t_1)) \geq d_\omega \geq d$ which contradicts to \eqref{global}. 
	
	We finally prove that $I(u(t))<0$ for all $t\in [0,T)$. Suppose it is not true, there exists $t_2 \in [0,T)$ such that $I(u(t_2)) \geq 0$. By the continuity of $t\mapsto I(u(t))$, there exists $t_3 \in (0,t_2]$ such that $I(u(t_3)) =0$. We have $K_\omega(u(t_3))<0, I(u(t_3))=0$, by the definition of $d_n$, we have $S_\omega(u(t_3)) \geq d_n \geq d$ which contradicts to \eqref{global}.
\end{proof}

\begin{proof}[ \bf {Proof of Theorem \ref{prop blowup}}]
	By the virial identity,
	\[
	\frac{d^2}{dt^2} \|xu(t)\|^2_{L^2} = 8 I(u(t)), \quad \forall t\in [0,T). 
	\]
	By the convexity argument, it suffices to show that there exists $\delta>0$ such that $I(u(t)) <-\delta$ for all $t\in [0,T)$. Since $K_-$ is invariant under the flow of \eqref{GP}, we have $K_\omega(u(t))<0$ and $I(u(t))<0$ for all $t\in [0,T)$. Fixed $t\in [0,T)$ and denote $u=u(t)$. For $\mu>0$, we set $u_\mu(x) = \mu^{\frac{N-b}{p+1}} u(\mu x)$. We have
	\begin{align*}
	K_\omega(u_\mu) &= \mu^{\frac{2(N-b)-(N-2)(p+1)}{p+1}} \|\nabla u\|^2_{L^2} + \mu^{\frac{2(N-b)-(N+2)(p+1)}{p+1}} \gamma^2 \|xu\|^2_{L^2} \\
	&\mathrel{\phantom{=}} + \mu^{\frac{2(N-b)-N(p+1)}{p+1}} \omega \|u\|^2_{L^2} - P(u),
	\end{align*}
	and
	\begin{align*}
	I(u_\mu)= \mu^{\frac{2(N-b)-(N-2)(p+1)}{p+1}} \|\nabla u\|^2_{L^2} &- \mu^{\frac{2(N-b)-(N+2)(p+1)}{p+1}} \gamma^2 \|xu\|^2_{L^2} \\
	&- \frac{N(p-1)+2b}{2(p+1)} P(u).
	\end{align*}
	Since $1+\frac{4-2b}{N} \leq p<2^{\circ}$, we see that the exponents of $\mu$ in $I(u_\mu)$ are positive and negative respectively. Since $I(u)<0$, it yields that there exists $\mu_0>1$ such that $I(u_{\mu_0}) =0$, and when $\mu \in [1,\mu_0), I(u_\mu)<0$. For $\mu \in [1, \mu_0]$, since $K_\omega(u)<0$, $K_\omega(u_\mu)$ has the following two possibilities:
	\begin{itemize}
		\item[a)] $K_\omega(u_\mu)<0$ for $\mu \in [1,\mu_0]$,
		\item[b)] there exists $1<\mu_1\leq \mu_0$ such that $K_\omega(u_{\mu_1})=0$. 
	\end{itemize}
	For the case a), we have $I(u_{\mu_0})=0$ and $K_\omega(u_{\mu_0})<0$. By the definition of $d_n$, we have $S_\omega(u_{\mu_0}) \geq d_n \geq d$. Moreover, we have 
	\begin{align*}
	S_\omega(u) - S_\omega(u_{\mu_0}) &= \frac{1}{2}\left(1-\mu_0^{\frac{2(N-b)-(N-2)(p+1)}{p+1}} \right) \|\nabla u\|^2_{L^2} \\
	&\mathrel{\phantom{=}} + \frac{1}{2} \left(1-\mu_0^{\frac{2(N-b)-(N+2)(p+1)}{p+1}} \right) \gamma^2 \|xu\|^2_{L^2} \\
	&\mathrel{\phantom{=}} + \frac{1}{2} \left(1-\mu_0^{\frac{2(N-b)-N(p+1)}{p+1}} \right) \omega \|u\|^2_{L^2},
	\end{align*}
	and
	\begin{align*}
	I(u)-I(u_{\mu_0}) &= \left(1- \mu_0^{\frac{2(N-b)-(N-2)(p+1)}{p+1}} \right) \|\nabla u\|^2_{L^2} \\
	&\mathrel{\phantom{=}} - \left( 1- \mu_0^{\frac{2(N-b)-(N+2)(p+1)}{p+1}} \right) \gamma^2 \|xu\|^2_{L^2}.
	\end{align*}
	Since $\mu_0>1$ and $1+ \frac{4-2b}{N} \leq p <2^{\circ}$, it follows that 
	\[
	S_\omega(u)-S_\omega(u_{\mu_0}) \geq \frac{1}{2} (I(u)-I(u_{\mu_0})) = \frac{1}{2} I(u).
	\]
	For the case b), we have $K_\omega(u_{\mu_1}) =0$ and $I(u_{\mu_1}) \leq 0$. By the definition of $d_\omega$, we have $S_\omega(u_{\mu_1}) \geq d_\omega \geq d$. By the same argument as above, we have
	\[
	S_\omega(u)-S_\omega(u_{\mu_1}) \geq \frac{1}{2} (I(u)- I(u_{\mu_1})) \geq \frac{1}{2} I(u).
	\]
	In both cases, we prove that
	\[
	I(u) < 2 (S_\omega(u) - d). 
	\]
	Since the above argument is independent of $t\in [0,T)$, we get $I(u(t)) < - \delta$ for all $t\in [0,T)$, where $\delta = 2 (d-S_\omega(u_0)) >0$. Thus we obtain the proof of statement i) of theorem.
	
Next we prove ii). In the case $1 <p<1+\frac{4-2b}{N}$, the global existence follows from the sharp Gagliardo-Nirenberg inequality. Therefore, we only consider the case $1 + \frac{4-2b}{N} \leq p <2^{\circ}$. 
	
	1) Let us consider the case $u_0 \in R_+$. Since $R_+$ is invariant under the flow of \eqref{GP}, we have $S_\omega(u(t))<d$ and $K_\omega(u(t))>0$ for any $t\in [0,T)$. Since $K_\omega(u(t))>0$, it follows that $\|u(t)\|^2_{H_\omega} > P(u(t))$. Thus,
	\[
	\left(\frac{1}{2}-\frac{1}{p+1}\right) \|u(t)\|^2_{H_\omega} < \frac{1}{2} \|u(t)\|^2_{H_\omega} - \frac{1}{p+1} P(u(t)) = S_\omega(u(t)) <d.
	\]
	We get $\|u(t)\|^2_{H_\omega} < \frac{2d(p+1)}{p-1}$ for any $t\in [0,T)$. Since $\|u\|^2_{H_\omega} \sim \|u\|^2_{\Sigma}$, this implies that the solution exists globally in time.
	
	2) Let us now consider the case $u_0 \in K_+$. Since $K_+$ is invariant under the flow of \eqref{GP}, we have $S_\omega(u(t))<d, K_\omega(u(t))<0$ and $I(u(t))>0$. It follows that
	\begin{multline*}
	\left(\frac{1}{2}-\frac{2}{N(p-1)+2b}\right)\|\nabla u(t)\|^2_{L^2} + \left( \frac{1}{2} + \frac{2}{N(p-1)+2b} \right) \gamma^2 \|xu(t)\|^2_{L^2} + \frac{\omega}{2} \|u(t)\|^2_{L^2} \\
	< \frac{1}{2} \|u(t)\|^2_{H_\omega} - \frac{1}{p+1} P(u(t)) = S_\omega(u(t)) <d.
	\end{multline*}
	
	In the case $L^2$-supercritical case $1+\frac{4-2b}{N} < p < 2^{\circ}$, it follows from the above inequality that $\|\nabla u(t)\|^2_{L^2} < C$ for some constant $C>0$ and for any $t \in [0,T)$. This shows that the solution exists globally in time.
	
	In the $L^2$-critical case $p=1+\frac{4-2b}{N}$, we have
	\begin{align}
	\gamma^2 \|xu(t)\|^2_{L^2} + \frac{\omega}{2} \|u(t)\|^2_{L^2} <d. \label{L2 glo}
	\end{align}
	Fixed $t\in [0,T)$ and denote $u=u(t)$. We set $u_\mu(x) = \mu^{\frac{N(N-b)}{2N+4-2b}} u(\mu x)$. A direct computation shows that
	\[
	I(u_\mu) = \mu^{\frac{4-2b}{N+2-b}} \|\nabla u\|^2_{L^2} - \mu^{-\frac{4N+4-2b}{N+2-b}} \gamma^2 \|xu\|^2_{L^2} - \frac{N(p-1)+2b}{2(p+1)} P(u).
	\]
	Thus, $I(u)>0$ implies that there exists $0<\mu_0 <1$ such that $I(u_{\mu_0}) =0$. It follows that
	\begin{align*}
	S_\omega(u_{\mu_0}) &= \gamma^2 \|xu_{\mu_0}\|^2_{L^2} +\frac{1}{2} \omega \|u_{\mu_0}\|^2_{L^2} \\
	&= \mu_0^{-\frac{4N+4 -2b}{N+2-b}} \gamma^2 \|xu\|^2_{L^2} + \mu_0^{-\frac{2N}{N+2-b}} \omega \|u\|^2_{L^2}.
	\end{align*}
	It follows from \eqref{L2 glo} that 
	\begin{align}
	S_\omega(u_{\mu_0}) < \mu_0^{-\frac{4N+4-2b}{N+2-b}} d. \label{S_ome glo}
	\end{align}
	We now consider $K_\omega(u_{\mu_0})$ which has two possibilities. The first one is $K_\omega(u_{\mu_0})<0$. By the definition of $d_n$ and the fact $I(u_{\mu_0})=0$, we have
	\[
	S_\omega(u_{\mu_0}) \geq d_n \geq d > S_\omega(u).
	\]
	It follows that
	\[
	S_\omega(u) - S_\omega(u_{\mu_0})<0, 
	\]
	which is
	\begin{align*}
	\left( 1- \mu_0^{\frac{4-2b}{N+2-b}} \right) \|\nabla u\|^2_{L^2}  + \left( 1- \mu_0^{-\frac{4N+4-2b}{N+2-b}}\right) \gamma^2 \|xu\|^2_{L^2} 
	+\left(1-\mu_0^{-\frac{2N}{N+2-b}}\right) \omega \|u\|^2_{L^2} <0.
	\end{align*}
	It implies that
	\[
	\|\nabla u\|^2_{L^2}  \lesssim \gamma^2 \|xu\|^2_{L^2} + \omega \|u\|^2_{L^2}.
	\]
	Thanks to \eqref{L2 glo}, we get $\|\nabla u\|^2_{L^2} < C$ for some constant $C>0$.
	
	The second possibility is that $K_\omega(u_{\mu_0}) \geq 0$. In this case, using \eqref{S_ome glo}, we have
	\[
	S_\omega(u_{\mu_0}) - \frac{1}{p+1} K_\omega(u_{\mu_0}) < \mu_0^{-\frac{4N+4-2b}{N+2-b}} d.
	\]
	It follows that 
	\begin{align*}
	\frac{p-1}{2(p+1)} \left( \mu_0^{\frac{4-2b}{N+2-b}} \|\nabla u\|^2_{L^2} + \mu_0^{-\frac{4N+4-2b}{N+2-b}} \gamma^2 \|xu\|^2_{L^2} + \mu_0^{-\frac{2N}{N+2-b}} \omega \|u\|^2_{L^2} \right) 
	< \mu_0^{-\frac{4N+4-2b}{N+2-b}} d. 
	\end{align*}
	We thus get $\|\nabla u\|^2_{L^2} < C$ for some constant $C>0$. In both possibilities, we always have the boundedness of $\|\nabla u\|^2_{L^2}$. Since the above argument is independent of $t\in [0,T)$, we obtain the boundedness of $\|\nabla u(t)\|^2_{L^2}$ for any $t\in [0,T)$. Therefore, the solution exists globally in time in the $L^2$-critical case $p=1+\frac{4-2b}{N}$. This completes the proof of theorem.
\end{proof}

\section{Normalized ground states}
\label{S:3}
This section is devoted to the proof of Theorems \ref{CMm} and \ref{CC} stated in the introduction. Before giving the proof of Theorem \ref{CMm} we recall that  the embedding $\Sigma (\mathbb{R}^{N})\hookrightarrow L^{r+1}(\mathbb{R}^{N})$ is compact, where $1\leq r<1+4/(N-2)$ ($N\geq 3$), $1\leq  r<\infty$ ($N=1$, $2$.); see  \cite[Lemma 3.1]{JZ2000}.

Now we give the proof of Theorem \ref{CMm}. 
\begin{proof}[ \bf {Proof of Theorem \ref{CMm}}]  
Let $\left\{u_{n}\right\}$ be a minimizing sequence for the problem $I_{q}$, then we have that $\left\{u_{n}\right\}$ is bounded in $\Sigma (\mathbb{R}^{N})$. Indeed, by the Gagliardo-Nirenberg's inequality \eqref{GNI} and Young's inequality we see that 
\begin{equation*}
\int_{\mathbb{R}^{N}}|x|^{-b}|u_{n}|^{p+1}dx\leq \epsilon\|\nabla u_{n}\|^{\left(\frac{N(p-1)}{2}+b\right)\alpha}_{L^{2}}+C_{\epsilon}\|u_{n}\|^{\left(p+1-\frac{N(p-1)}{2}-b\right)\beta}_{L^{2}},
\end{equation*}
where $C_{\epsilon}>0$ and $1/\alpha+1/\beta=1$. Now choosing $\alpha=\frac{4}{N(p-1)+2b}>1$ (it is due to the assumption $1<p<\frac{4-2b}{N}$), it follows that 
\begin{equation*}
\int_{\mathbb{R}^{N}}|x|^{-b}|u_{n}|^{p+1}dx\leq \epsilon\|\nabla u_{n}\|^{2}_{L^{2}}+C_{\epsilon}q^{\left(p+1-\frac{N(p-1)}{2}-b\right)\beta}.
\end{equation*}
Eventually, we get
\begin{equation*}
E(u_{n})\geq \left(\frac{1}{2}-\frac{\epsilon}{p+1}\right)\|\nabla u_{n}\|^{2}_{L^{2}}+\frac{\gamma^{2}}{2}\|x\, u_{n}\|^{2}_{L^{2}}-\frac{C_{\epsilon}}{p+1}q^{\left(p+1-\frac{N(p-1)}{2}-b\right)\beta}.
\end{equation*}
Taking $\epsilon>0$ sufficiently small, this implies  that $\left\{u_{n}\right\}$ is bounded in $\Sigma (\mathbb{R}^{N})$. Therefore, there exists $u\in \Sigma (\mathbb{R}^{N})$ such that, up to a subsequence, we can suppose that $u_{n}$ converges to $u$ weakly in $\Sigma (\mathbb{R}^{N})$. Since $\Sigma (\mathbb{R}^{N})\hookrightarrow L^{r+1}(\mathbb{R}^{N})$ is compact, it folows that  $u_{n}\rightarrow u$ in $L^{r+1}$ for $1\leq r<1+4/(N-2)$ if $N\geq 3$ and $1\leq r<\infty$ if $N=1$, $2$. 
By \eqref{conver}, we have $P(u_n) \rightarrow P(u_0)$ as $n\rightarrow \infty$. From the lower semi continuity we have  
\begin{equation*}
\|\nabla u\|^{2}_{L^{2}}+\frac{\gamma^{2}}{2}\|x\, u\|^{2}_{L^{2}}\leq \liminf_{n\rightarrow \infty}\left\{\|\nabla u_{n}\|^{2}_{L^{2}}+\frac{\gamma^{2}}{2}\|x\, u_{n}\|^{2}_{L^{2}}\right\}.
\end{equation*}
It follows that $E(u)\leq \liminf_{n\rightarrow \infty}E(u_{n})$ and $\|u\|^{2}_{L^{2}}=q$, which implies that $u$ is a minimizer of $I_{q}$ and $E(u)= \lim_{n\rightarrow \infty}E(u_{n})$; consequently $u_{n}\rightarrow u$ in $\Sigma (\mathbb{R}^{N})$ as $n\rightarrow +\infty$ and $u\in \mathcal{G}_{q}$, which completes the proof of Item (i). By the same argument as in the Theorem \ref{prop GS} we get that there exists a positive and spherically symmetric function such that $u(x)=e^{i\theta_{0}}\varphi(x)$. This concludes the proof.
\end{proof}

\begin{proof}[ \bf {Proof of Theorem \ref{CC}}]  
Let $p=1+\frac{4-2b}{N}$. Assume that $\left\{u_{n}\right\}$ is a minimizing sequence for $I_{q}$ with $q<\|Q\|^{2}_{L^{2}}$. Then $\left\{u_{n}\right\}$ is bounded in $\Sigma (\mathbb{R}^{N})$. Indeed, since $E(u_{n})\leq I_{q}+1$ for $n$ sufficiently large, by \eqref{Lk} we infer that 
\begin{equation*}
\frac{1}{2}\|\nabla u_{n}\|^{2}_{L^{2}}\left(1-\left(\frac{q}{\|Q\|_{L^{2}}}\right)^{\frac{4-2b}{N}}\right)+\frac{\gamma^{2}}{2}\int_{\mathbb{R}^{N}}|x|^{2}| u_{n}|^{2}dx\leq I_{q}+1.
\end{equation*}
Therefore, we have that $\|u_{n}\|_{\Sigma(\mathbb{R}^{N})}$ is bounded. Thus there exists $u\in \Sigma(\mathbb{R}^{N})$ such that $u_{n}\rightharpoonup u$ in $\Sigma(\mathbb{R}^{N})$ and $u_{n}\rightarrow u$ in $L^{r+1}$ for $1\leq r<1+4/(N-2)$ if $N\geq 3$ and $1\leq r<\infty$ if $N=1$, $2$, as $n$ goes to $+\infty$. From here, the proof of Theorem \ref{CC}  is completed exactly as the proof of Theorem \ref{CMm}.
\end{proof}

\section{The Supercritical Case}
\label{S:4}
In this section, we prove  Theorem \ref{Lmp}.  Firstly we give
\begin{lemma}\label{Mlh} 
Let $1+\frac{4-2b}{N}<p<1+\frac{4-2b}{N-2}$.  The following facts hold:\\
(i) $S_{q}\cap B_{r}$ is not empty set iff $q\leq \frac{r}{\gamma N}$.\\
(ii) For any $r>0$, there exists $q_{0}=q_{0}(r)$ such that, for every $q<q_{0}$,
\begin{equation}\label{II}
\inf\left\{E(u),\quad u\in S_{q}\cap B_{rq/2}\right\}< \inf\left\{E(u),\quad u\in S_{q}\cap (B_{r}\setminus B_{rq} )\right\}
\end{equation}
\end{lemma}
\begin{proof}
 We set $\zeta(x):=\sqrt{q}\Phi(x)$, where $\Phi$ is given in \eqref{Efa}. For any $r>0$, if $q\leq \frac{r}{\gamma N}$, it is clear that 
\begin{equation*}
\|\zeta\|^{2}_{L^{2}}=q \quad \text{and} \quad \|\zeta\|^{2}_{H}=\|\nabla\zeta\|^{2}_{L^{2}}+\gamma^{2}\|x\,\zeta\|^{2}_{L^{2}}=\gamma N\|\zeta\|^{2}_{L^{2}}\leq r.
\end{equation*}
Here, the norm $\|\cdot\|^{2}_{H}$ is defined in \eqref{Hn}. Therefore $\zeta\in S_{q}\cap B_{r}$. On the other hand, if $u\in S_{q}\cap B_{r}$, then from \eqref{Igst} we infer
\begin{equation*}
r\geq \|u\|^{2}_{H}=\|\nabla u\|^{2}_{L^{2}}+\gamma^{2}\|x\,u\|^{2}_{L^{2}}\geq \gamma N q,
\end{equation*}
which completes the proof of the statement (i) above. 

Our proof of statement (ii)  is inspired by the one of Lemma 3.1 in \cite{BEBOJEVI2017}. From Gagliardo-Nirenberg inequality \eqref{GNI} we get 
\begin{equation}\label{Ic1}
\begin{cases} 
E(u)\geq \frac{1}{2}\| u\|^{2}_{H}-Cq^{\frac{p+1}{2}-\frac{N(p-1)}{4}-\frac{b}{2}}\| u\|^{\frac{N(p-1)}{2}+{b}}_{H}=\alpha_{q}(\| u\|_{H}),\\
E(u)\leq \frac{1}{2}\| u\|^{2}_{H}=\beta_{q}(\| u\|_{H}),
\end{cases} 
\end{equation}
where 
\begin{equation*} 
\begin{cases}
 \alpha_{q}(t)=\frac{1}{2}t(1-2Cq^{\chi}t^{\delta})\\
\beta_{q}(t)=\frac{1}{2}t
\end{cases} 
\end{equation*}
and
\begin{equation*} 
 \chi=\frac{1}{2}\left(p+1-\frac{N(p-1)}{2}-b\right)>0, \quad \delta=\frac{N(p-1)+2b-4}{4}>0.
\end{equation*}
Note that, by \eqref{Ic1}, to prove \eqref{II}, it suffices to show that there exists $0<q_{0}=q_{0}(r)<<1$ such that, for every $q<q_{0}$,
\begin{equation*} 
\beta_{q}(qr/2) <\inf_{t\in (rq,r)}\alpha_{q}(t).
\end{equation*}
Now since $\alpha_{q}(t)>\frac{5}{16}t$ for $t\in (0, r)$ and $q<q_{0}(r)<<1$, we get
\begin{equation*} 
\beta_{q}(qr/2)=\frac{1}{4}qr<\frac{5}{16}qr\leq\inf_{t\in (rq,r)}\alpha_{q}(t), 
\end{equation*}
which completes the proof of lemma.
\end{proof}

\begin{proof}[ \bf {Proof of Theorem \ref{Lmp}}]  
Suppose that $\left\{u_{n}\right\}$ is a minimizing sequence for $I^{r}_{q}$. Since $\left\{u_{n}\right\}\subset S_{q}\cap B_{r}$, it follows that $\left\{u_{n}\right\}$ is bounded in $\Sigma(\mathbb{R}^{N})$. Then there exists $u\in \Sigma(\mathbb{R}^{N})$ such that $u_{n}\rightharpoonup u$ in $\Sigma(\mathbb{R}^{N})$ and $u_{n}\rightarrow u$ in $L^{2}$ as $n\rightarrow\infty$.
By lower semi-continuity 
\begin{equation*}
\|\nabla u\|^{2}_{L^{2}}+\gamma^{2}\|x\, u\|^{2}_{L^{2}}\leq \liminf_{n\rightarrow \infty}\left\{\|\nabla u_{n}\|^{2}_{L^{2}}+\gamma^{2}\|x\, u_{n}\|^{2}_{L^{2}}\right\},
\end{equation*}
we infer that $u\in S_{q}\cap B_{r}$ and $E(u)\leq \lim_{n\rightarrow \infty} E(u_{n})=I_{q}$. Thus, $u\in \mathcal{G}_{q}$ and $u_{n}\rightarrow u$ in $\Sigma(\mathbb{R}^{N})$. Moreover, by the same argument as in the proof of the Theorem \ref{CMm}, we see that there exist  a real-valued positive function $\varphi$ and $\theta\in \mathbb{R}$ such that $u=e^{i\theta}\varphi$. 

Now, since $\|\varphi^{\ast}\|^{2}_{L^{2}}=\|\varphi\|^{2}_{L^{2}}$ and $\|\nabla\varphi^{\ast}\|^{2}_{L^{2}}\leq\|\nabla\varphi\|^{2}_{L^{2}}$, from  \eqref{S1} we have that $\varphi^{\ast}\in S_{q}\cap B_{r}$. In addition, if we suppose that $\varphi$ is not radial, then by \eqref{S1}-\eqref{S2} we infer that $E(\varphi^{\ast})<E(\varphi)$, which is a contradiction. Therefore $\varphi$ is  radial and radially decreasing,
which completes the proof of the statements (i) and (iii).

Now we prove statement (ii). From Lemma \ref{Mlh} we infer that $\varphi\in B_{rq}$. This implies that $\varphi$ does not belong to the boundary of $S_{q}\cap B_{r}$. Then, we have that $\varphi$ is a critical point of $E$ on $S_{q}$ and there exists a Lagrange multiplier $\omega\in \mathbb{R}$ such that the Euler-Lagrange equation
\begin{equation}\label{Spa}
-\Delta \varphi+\omega\varphi+\gamma^{2}|x|^{2}\varphi-|x|^{-b}|\varphi|^{p-1}\varphi=0,
\end{equation}
holds. 

Let $\zeta$ be the eigenfunction defined in \eqref{Efa} such that $\|\zeta\|^{2}_{L^{2}}=q$. Then $\zeta\in S_{q}\cap B_{r}$ and
\begin{equation*}
I^{r}_{q}\leq E(\zeta)=\frac{1}{2}\gamma N q-\frac{1}{p+1}\int_{\mathbb{R}^{N}}|x|^{-b}|\zeta|^{p+1}dx<\frac{1}{2}\gamma N q.
\end{equation*}
Thus, from \eqref{Spa} we see that
\begin{align}\nonumber
\omega \|\varphi\|^{2}_{L^{2}}&=- \|\nabla \varphi\|^{2}_{L^{2}}-\gamma^{2}\|x\, \varphi\|^{2}_{L^{2}}+\int_{\mathbb{R}^{N}}|x|^{-b}|\varphi|^{p+1} dx\\\label{Pe1}
&=-2I^{r}_{q}+\frac{p-1}{p+1}\int_{\mathbb{R}^{N}}|x|^{-b}|\varphi|^{p+1} dx> -\gamma N q.
\end{align}
Therefore $\omega>-\gamma N$. Now, from \eqref{GNI}  we obtain

\begin{align*}
\omega \|\varphi\|^{2}_{L^{2}}&= -\|\nabla \varphi\|^{2}_{L^{2}}-\gamma^{2}\|x\, \varphi\|^{2}_{L^{2}}+\int_{\mathbb{R}^{N}}|x|^{-b}|\varphi|^{p+1} dx\\
&\leq -\|\varphi\|^{2}_{H}+C\|\varphi\|^{\frac{N(p-1)}{2}+b}_{H}q^{\frac{p+1}{2}-\frac{N(p-1)}{4}-\frac{b}{2}}\\
&=-\|\varphi\|^{2}_{H}\left(1-C\|\varphi\|^{\frac{N(p-1)}{2}+b-2}_{H}q^{\frac{p+1}{2}-\frac{N(p-1)}{4}-\frac{b}{2}}\right)\\
&\leq -\|\varphi\|^{2}_{H}\left(1-C(rq)^{\frac{N(p-1)}{4}+\frac{b}{2}-1}q^{\frac{p+1}{2}-\frac{N(p-1)}{4}-\frac{b}{2}}\right)\\
&\leq-\|\varphi\|^{2}_{H}\left(1-Cq^{\frac{p-1}{2}}\right),
\end{align*}
and with \eqref{Igst} we obtain 
\begin{equation*}
\omega \leq -\gamma N\left(1-Cq^{\frac{p-1}{2}}\right).
\end{equation*}
This completes the proof of theorem.
\end{proof}

\section{Orbital stability}
\label{S:5}
This section is devoted to the proof of Theorems \ref{Et} and \ref{INSF}. 

\begin{proof}[ \bf {Proof of Theorem \ref{Et}}]  
We only consider the supercritical case $1+\frac{4-2b}{N}<p<1+\frac{4-2b}{N-2}$, the proof in the other cases, when $1<p\leq 1+\frac{4-2b}{N}$, is similar.
We verify the statement of Theorem \ref{Et} (iii) by contradiction. Then we have that there exist $\epsilon>0$ and two sequences $\left\{u_{0,n}\right\}\subset \Sigma(\mathbb{R}^{N})$ and $\left\{t_{n}\right\}\subset \mathbb{R}$ such that 
\begin{align}\label{L3}
&\inf_{\varphi\in \mathcal{G}^{r}_{q}}\|u_{0, n}-\varphi\|_{\Sigma(\mathbb{R}^{N})}\rightarrow 0 \quad \text{as $n\rightarrow +\infty$} \\\label{L4}
&\sup_{t\in \mathbb{R}}\inf_{\varphi\in \mathcal{G}^{r}_{q}}\|u_{n}(t_{n})-\varphi\|_{\Sigma(\mathbb{R}^{N})}\geq \epsilon\quad \text{ for every $n\in\mathbb{N}$.} 
\end{align}
Here $u_{n}(t)$ is the maximal solution of \eqref{GP} with initial datum $u_{0, n}$. Without loss of generality, we may assume that $\|u_{0, n}\|^{2}_{L^{2}}=q$. From \eqref{L3} and the conservation of charge and energy we infer that
\begin{align*}
& \|u_{n}(t_{n})\|^{2}_{L^{2}}=\|u_{0, n}\|^{2}_{L^{2}}=q  \quad \text{ for every $n$,} \\
& E(u_{n}(t_{n}))=E(u_{0, n})\rightarrow I^{r}_{q} \quad \text{ as $n\rightarrow +\infty$.}
\end{align*}
We claim that there exists a  subsequence $\left\{u_{n_{k}}(t_{n_{k}})\right\}$ of $\left\{u_{n}(t_{n})\right\}$ such that $\|u_{n_{k}}(t_{n_{k}})\|^{2}_{H}\leq r$. Indeed, suppose that there exists $K\geq 1$ such that  $\|u_{n}(t_{n})\|^{2}_{H}> r$ for every $n\geq K$. By continuity, there exists $t_{n}^{\ast}\in (0, t_{n})$ such that $\|u_{n}(t^{\ast}_{n})\|^{2}_{H}= r$.  Since $\|u_{n}(t^{\ast}_{n})\|^{2}_{L^{2}}=q$, $\|u_{n}(t^{\ast}_{n})\|^{2}_{H}= r$ and $E(u_{n}(t^{\ast}_{n}))=E(u_{0, n})\rightarrow I^{r}_{q}$ as $n\rightarrow +\infty$, it follows that
$\left\{u_{n}(t^{\ast}_{n})\right\}$ is a minimizing sequence of $I^{r}_{q}$. From Theorem \ref{Lmp}, we infer that there exists $\psi\in \Sigma(\mathbb{R}^{N})$ such that $\|\psi\|^{2}_{L^{2}}=q$, $\|\psi\|^{2}_{H}= r$ and $E(\psi)=I^{r}_{q}$, which is a contradiction with Lemma \ref{Mlh} (ii), because the critical point $\psi$ does not belong to the boundary of $S_{q}\cap B_{r}$. Therefore, there exists a  subsequence $\left\{u_{n_{k}}(t_{n_{k}})\right\}$ such that $\|u_{n_{k}}(t_{n_{k}})\|^{2}_{H}\leq r$ for all $k\geq1$. In particular, $\left\{u_{n_{k}}(t_{n_{k}})\right\}$ is  a minimizing sequence for $I^{r}_{q}$. Again from Theorem \ref{Lmp} we obtain, passing to a subsequence if necessary,
\begin{equation*}
\inf_{\varphi\in \mathcal{G}^{r}_{q}}\|u_{n_{k}}(t_{n_{k}})-\varphi\|_{\Sigma(\mathbb{R}^{N})}\rightarrow 0 \quad \text{as $k\rightarrow +\infty$}, 
\end{equation*}
which is a contradiction with \eqref{L4} and finishes the proof.
\end{proof}

 Next we study the instability of standing waves for \eqref{GP} in the $L^2$-critical and $L^2$-supercritical cases. 

\begin{proof}[ \bf {Proof of Theorem \ref{INSF}}]
	Since $d_n \geq d_\omega$, we have $d=d_\omega$. From Lemma \ref{lem poho}, $K_\omega(\varphi) = I(\varphi) =0$. Set $\varphi^\lambda(x) = \lambda \varphi(x)$. Since 
	\begin{align*}
	K_\omega(\varphi^\lambda) &= \lambda^2 \|\varphi\|^2_{H_\omega} - \lambda^{p+1} P(\varphi), \\
	I(\varphi^\lambda) &= \lambda^2 (\|\nabla \varphi\|^2_{L^2} - \gamma^2 \|x\varphi\|^2_{L^2}) - \frac{N(p-1)+2b}{2(p+1)} \lambda^{p+1} P(\varphi),
	\end{align*}
	it is easy to see that the equations $K_\omega(\varphi^\lambda) =0$ and $I(\varphi^\lambda)=0$ have unique non-zero solution $\lambda_0=1$. It follows that for any $\lambda>1$, 
	\[
	K_\omega(\varphi^\lambda)<0, \quad I(\varphi^\lambda)<0. 
	\]
	On the other hand, we notice that $\frac{d}{d\lambda} S_\omega(\varphi^\lambda) = \lambda^{-1} K_\omega(\varphi^\lambda)$. Thus, $S_\omega(\varphi^\lambda) < S_\omega(\varphi)$ for any $\lambda>1$. Since $S_\omega(\varphi) = d_\omega =d$, we see that for any $\lambda>1$, $S_\omega(\varphi^\lambda)<d, K_\omega(\varphi^\lambda) <0, I_\omega(\varphi^\lambda)<0$. This implies that $\varphi^\lambda \in K_-$ for any $\lambda>1$. Now let $\epsilon>0$. We take $\lambda_1>1$ sufficiently close to 1 such that
	\[
	\|\varphi^{\lambda_1} - \varphi\|_{\Sigma} = (\lambda_1 - 1) \|\varphi\|_{\Sigma} < \epsilon.
	\] 
	Set $u_0 = \varphi^{\lambda_1}$, we see that $u_0 \in K_-$. By Proposition \ref{prop blowup}, the corresponding solution blows up in finite time. Thus we obtain the proof of statement i) of theorem.

Next we prove ii). In this case $d= d_n< d_\omega$. Since $u\in \mathcal{M}_{\omega}$, we have
	\[
	S_\omega(\varphi^\lambda) < S_\omega(\varphi) = d_\omega
	\]
	for any $\lambda>1$. Since $\frac{d}{d\lambda}S_\omega(\varphi^\lambda) = \lambda^{-1} K_\omega(\varphi^\lambda)$ and since $K_\omega(\varphi) =0$, we have $\frac{d}{d\lambda} S_\omega(\varphi^\lambda)<0$ for any $\lambda>1$. On the other hand, $S_\omega(\varphi) =d, S_\omega(\varphi^\lambda) \rightarrow -\infty$ as $\lambda \rightarrow \infty$. Thus, there exists $\lambda_0>1$ such that $S_\omega (\varphi^\lambda) < S_\omega(\varphi^{\lambda_0}) = d$ as $\lambda>\lambda_0$. It follows that $S_\omega(\varphi^\lambda) <d, K_\omega(\varphi^\lambda) <0, I(\varphi^\lambda)<0$ for any $\lambda>\lambda_0$ or $\varphi^\lambda \in K_-$ for any $\lambda>\lambda_0$. Taking $\delta = (\lambda_0 - 1) \|\varphi\|_{\Sigma}$ and choose $u_0 = \varphi^{\lambda_1}$ for some $\lambda_1>\lambda_0$, the result follows.
\end{proof}

\section{Appendix}
\label{S:8}
In this appendix we show the uniqueness result for \eqref{Sp}. More specifically,  if $N\geq 3$,  $0 < b < 1$  and $1 < p < 1 + \frac{4-2b}{N-2}$, then for any $\omega> -\gamma\, N$ there exists a unique positive radial solution of \eqref{Sp}.
 
Through this appendix we assume that $N\geq 3$,  $0 < b < 1$  and $1 < p < 1 + \frac{4-2b}{N-2}$. 

In \cite[Theorem 1]{NSKW}, Shioji and Watanabe  give a uniqueness result for positive radial solutions of
\begin{equation*}
\varphi^{\prime\prime}+\frac{N-2}{r}\varphi^{\prime}+g(r)\varphi+h(r)\varphi^{p}=0 \quad \text{in $(0, +\infty)$,}
\end{equation*}
under  appropriate assumptions on $g(r)$ and $h(r)$. Note that for our case, Eq \eqref{Sp}, we have that $g(r)=-(\omega+\gamma^{2}r^{2})$ and $h(r)=r^{-b}$.

Required conditions in \cite[Theorem 1]{NSKW} are following.\\
(I) $g\in C^{1}((0, +\infty))$, $h\in C^{3}((0, +\infty))$; $g(r)>0$, $h(r)>0$ for every $r\in (0, +\infty)$.\\
(II) $\lim_{r\rightarrow 0}\frac{1}{r^{N-1}}\int^{r}_{0}\tau^{N-1}\left[g(\tau)+h(\tau)\right]\,d\tau=0.$\\
(III) There exists $r^{\ast}\in (0, +\infty)$ such that\\
(i) $r^{N-1}g(r)\in L^{1}((0, r^{\ast}))$, $r^{N-1}h(r)\in L^{1}((0, r^{\ast}))$.\\
(ii) $\tau^{N-1}\left(g(\tau)+h(\tau)\right)\left(\frac{(r^{\ast})^{2-N}-\tau^{2-N}}{2-N}\right)\in  L^{1}((0, r^{\ast}))$.\\
(IV) $\lim_{r\rightarrow 0}a(r)<+\infty$, $\lim_{r\rightarrow 0}|\beta(r)|<+\infty$, $\lim_{r\rightarrow 0}c(r)\in [0,+\infty]$, $\lim_{r\rightarrow 0}a(r)g(r)=0$ and  $\lim_{r\rightarrow 0}a(r)h(r)=0$, where
\begin{align*}
a(r)=& r^{\frac{2(N-1)(p+1)}{p+3}}h(r)^{-\frac{2}{p+3}}, \\
\beta(r)=& -\frac{1}{2}a^{\prime}(r)+\frac{N-1}{r}a(r),\\
c(r)=& -\beta^{\prime}(r)+\frac{N-1}{r}\beta(r).
\end{align*}
(V) There exists $k\in [0,+\infty)$ such that
\begin{align*}
G(r)> 0 \,\,\, \text{on $(0, k)$ and }\,\,\, G(r)< 0\,\, \text{ on $(k, +\infty)$,}
\end{align*}
where 
\begin{align*}
G(r)=-\beta(r)g(r)+\frac{1}{2}c^{\prime}(r)+\frac{1}{2}(ag)^{\prime}(r).
\end{align*}
(VI) $G^{-}\neq 0$ is satisfied, where $G^{-}=\min\left\{G(r),0\right\}$ for $r\in (0, +\infty)$.

Next we check the conditions (I)-(VI) to prove the uniqueness of a solution of \eqref{Sp}. Since $N\geq 3$ and  $0 < b < 1$,  it is clear  that the conditions (I)-(III) hold true. For simplicity, we assume that $\gamma=1$. Recalling that $g(r)=-(\omega+r^{2})$ and $h(r)=r^{-b}$, a straightforward calculations give 
\begin{align*}
a(r)=& r^{2\frac{b+(N-1)(p+1)}{p+3}}, \\
\beta(r)=& \frac{1}{p+3}r^{\frac{2b+2N(p+1)-3p-5}{p+3}}\left(2(N-1)-b\right),\\
c(r)=& \frac{1}{(p+3)^{2}}r^{\frac{2(b+N(p+1)-2(p+2))}{p+3}}\left(b-2-(N-1)\right)\left(2b+N-(p-1)-2(p+1)\right),
\end{align*}
and 
\begin{align*}
G(r)=A\,r^{2}+B\,r+C,
\end{align*}
where
\begin{align*}
A=& -(p+3)^{2}(2b+N(p-1)+4) \\
B=& \omega(p+3)^{2}(2N-(2+b))\\
C=& (b-2N+2)(p(N-2)+b+N-4)(p(N-2)+2b-N-2).
\end{align*}
Since $N\geq 3$,  $0 < b < 1$  and $1 < p < 1 + \frac{4-2b}{N-2}$,  it is not hard to show that  (IV)-(VI) hold true. In particular, we obtain $A<0$ and $C\geq0$, thus we can find that there exists $k\in [0,+\infty)$ such that $G(r)> 0$ on $(0, k)$ and  $G(r)< 0$ on $(k, +\infty)$. Hence by \cite[Theorem 1]{NSKW} we see that  there exists a unique positive radial solution of \eqref{Sp}.


\end{document}